\documentclass[12pt]{article}
\usepackage[]{lineno}

\usepackage{mathrsfs}
\usepackage{amsmath,extpfeil}
\usepackage{stmaryrd}
\usepackage{mathrsfs}
\usepackage{url}
\usepackage{indentfirst}
\usepackage{enumerate}
\usepackage{amsmath,amsfonts,amssymb,amsthm}
\usepackage{amsmath,amssymb,amsthm,amscd}
\usepackage{graphicx,mathrsfs}
\usepackage{appendix}
\usepackage[numbers,sort&compress]{natbib}
\usepackage{color}
\usepackage[colorlinks, linkcolor=blue, citecolor=blue]{hyperref}

\usepackage{sidecap}
\usepackage{float}
\usepackage{extarrows}
\usepackage{booktabs}
\usepackage{verbatim}
\usepackage[usenames,dvipsnames]{xcolor}

\newtheorem{theorem}{Theorem}[section]
\newtheorem{lemma}[theorem]{Lemma}

\newtheorem{proposition}[theorem]{Proposition}

\theoremstyle{definition}
\newtheorem{remark}[theorem]{Remark}

\def\XXint#1#2#3{{\setbox0=\hbox{$#1{#2#3}{\int}$}
         \vcenter{\hbox{$#2#3$}}\kern-.5\wd0}}

\numberwithin{equation}{section}

\allowdisplaybreaks[4]

\begin{document}

\title{Traveling waves of NLS system arising in optical material without Galilean symmetry}

\author{ \bf
	{\large Yuan Li$^1$, Kai Wang$^2$   and  Qingxuan Wang$^3$\footnote{Corresponding author:  Qingxuan Wang}
     \footnote{Email: li\_yuan@lzu.edu.cn (Y. Li), wangkai03@tyut.edu.cn(K. Wang), wangqx@zjnu.edu.cn(Q. Wang) }}\\
	{\small \emph{$^1$School of Mathematics and Statistics, Lanzhou University, Lanzhou, China}}\\
    {\small \emph{$^2$College of Mathematics, Taiyuan University of Technology, Taiyuan,   China}}\\
	{\small \emph{$^3$School of Mathematical Sciences, Zhejiang Normal University, Jinhua,  China}}
}

\date{}
\maketitle

\begin{abstract}
We study a nonlinear Schr\"odinger (NLS) system with cubic interactions arising in nonlinear optics, focusing on how the absence of Galilean symmetry affects solution dynamics. In the mass-critical case, we establish a new global existence result that contrasts with classical results. By exploiting the properties of traveling solitary waves, we obtain a condition on the initial data that relaxes the standard sharp assumptions. We also characterize the high-frequency limit of these traveling solitary waves as the frequency $\omega$ tends to infinity.

\noindent \textbf{Keywords:} Schr\"odinger system; Galilean symmetry; Global existence; High frequency limit

\noindent{\bf 2020 Mathematics Subject Classification:} Primary 35Q55; Secondary 35C07, 35B40
\end{abstract}

\section{Introduction and Main Results}
\noindent

In this paper, we consider the Cauchy problem for the following  nonlinear Schr\"odinger (NLS) system with cubic-interaction
\begin{equation}\label{equ:system}
    \begin{cases}
     i\partial_tu+\Delta u- u+\left(\frac{1}{9}|u|^2+2|v|^2\right)u+\frac{1}{3}\bar{u}^2v=0,\\
     i\gamma\partial_tv+\Delta v-\alpha v+(9|v|^2+2|u|^2)v+\frac{1}{9}u^3=0,
    \end{cases}
\end{equation}
with the initial datum $(u,v)|_{t=0}=(u_0,v_0)$.
Here $u,v:\mathbb{R}\times \mathbb{R}^N\to\mathbb{C}$ and the parameters $\gamma,\alpha>0$,  $\bar{u}$ is the complex conjugate of $u$.

System \eqref{equ:system} arises from a system of nonlinear Schr\"odinger equations in a suitable dimensionless form, see \cite{SBK1997optics}, where the resonant interaction between the beam of the fundamental frequency and its third harmonic in a diffractive dielectric waveguide is characterized by a single transverse dimension.
This system is also used to model cascading nonlinear processes. These processes can indeed generate effective higher-order nonlinearities, and they have stimulated the study of spatial solitary waves in optical materials with $\chi^{(2)}$ or $\chi^{(3)}$ susceptibilities.  For more details about the physical meaning of this system, one can refer to \cite{CDS2016nonlinearity,SBK1998JOSA} and the references therein. Recently, the study of system \eqref{equ:system} has attracted a quantity of attention in mathematics, with topics covering existence, scattering, blowup,   and stability, etc. see \cite{P2010EJDE,OP2021AMP,ZD2023MMAS,ADF2021CPDE,CW2023JMP}.

Let us review some basic properties of system \eqref{equ:system}. It is well known that the Cauchy problem \eqref{equ:system} is an infinite-dimensional Hamiltonian system and is locally well-posed in $H^1(\mathbb{R}^N)\times H^1(\mathbb{R}^N)$, $1\leq N\leq3$ (see \cite{Book2003NLS}). This system possesses the following three conservation laws:

Mass:
\begin{align}\label{mass}
M_{3\gamma}(u(t),v(t))=\|u\|_{L^2}^2+3\gamma\|v\|_{L^2}^2=M_{3\gamma}(u_0,v_0);
\end{align}

Energy:
\begin{align}\label{energy}
    E(u(t),v(t))=\frac{1}{2}\left(K(u(t),v(t))+M_{\alpha}(u(t),v(t))\right)-D(u(t),v(t))=E(u_0,v_0),
\end{align}
where
\begin{align}\notag
    &K(f,g)=\|\nabla f\|_{L^2}^2+\|\nabla g\|_{L^2}^2,~~M_{\alpha}(u(t),v(t))=\|u\|_{L^2}^2+\alpha\|v\|_{L^2}^2,\\\label{def:D}
    &D(f,g)=\int \frac{1}{36}|f(x)|^4+\frac{9}{4}|g(x)|^4+|f(x)|^2|g(x)|^2+\frac{1}{9}\Re\left(\bar{f}^3(x)g(x)\right);
\end{align}

Momentum:
\begin{align}\label{momentum}
    P(u(t),v(t))=(i\nabla u(t),u(t))+\gamma(i\nabla v(t),v(t))=P(u_0,v_0),
\end{align}
where
\[(u,v)=\Re \int u\bar{v}.\]
Notice that for $\gamma=3$, which is called the {\bf mass resonance} condition, system \eqref{equ:system} is invariant under the Galilean transformation
\begin{align*}
    (u,v)\mapsto\left(e^{\frac{i}{2}c\cdot x-\frac{i}{4}|c|^2t}u(t,x-ct),e^{\frac{i\gamma}{2}c\cdot x-\frac{i\gamma}{4}|c|^2t}v(t,x-ct)\right).
\end{align*}
for any $c\in\mathbb{R}^N$. If $\gamma\neq3$, the system \eqref{equ:system} is not invariant under the Galilean transformations.

In this paper, we will study an interesting result. This result is the global existence of \eqref{equ:system} for the mass-critical case $(N=2)$. When $N=2$, from \cite{OP2021AMP}, it is known that if the initial data $(u_0,v_0)\in H^1(\mathbb{R}^2)\times H^1(\mathbb{R}^2)$ and
\begin{align}\label{condition:global}
    M_{3\gamma}(u_0,v_0)<M_{3\gamma}(Q,P),~~\alpha=3\gamma,
\end{align}
 then the solution of \eqref{equ:system} is globally. Here $(Q,P)$ is the positive ground state solution of the following system:
\begin{align*}
    \begin{cases}
    -\Delta u+u=\frac{1}{9}u^3+2v^2u+\frac{1}{3}u^2v,\\
    -\Delta v+3\gamma v=9v^3+2u^2v+\frac{1}{9}u^3.
    \end{cases}
\end{align*}
From the sharp Gagliardo-Nirenberg inequality
\begin{align}\label{GN:1}
    D(u,v)\leq C_{opt}^{(1)} K(u,v)^{\frac{N}{2}}M_{3\gamma}(u,v)^{2-N/2},~~(u,v)\in H^1(\mathbb{R}^N)\times H^1(\mathbb{R}^N),
\end{align}
where
\begin{align*}
    C_{opt}^{(1)}= \frac{(4-N)^{\frac{N-2}{2}}}{N^{\frac{N}{2}}}\frac{1}{M_{3\gamma}(Q,P)},
\end{align*}
we know that the condition \eqref{condition:global} is sharp in the case $\gamma=3$. Indeed, there exist $u_0,v_0\in H^1(\mathbb{R}^2)$ satisfying
\begin{align*}
    M_{3\gamma}(u_0,v_0)=M_{3\gamma}(Q,P),~~\alpha=3\gamma=9,
\end{align*}
such that the system \eqref{equ:system} has the blowup solution, see \cite[Theorem 4.2]{OP2021AMP}. It was also proved that there exists a blowup solution in the weight space $L^2(\mathbb{R}^2,|x|^2dx)$; see \cite[Theorem 4.7 and Theorem 4.8]{OP2021AMP}.

{\bf Therefore, one question is whether the condition \eqref{condition:global} is still optimal when $\gamma\neq3$?}

In what follows, we will provide a conclusion (see Theorem \ref{Thm2} below) that the condition \eqref{condition:global} is not sharp under conditions $\gamma\neq3$ and $\alpha=3\gamma$.

To show this,  we first give the following sharp Gagliardo-Nirenberg inequality (see Appendix A), for $(u,v)\in H^1(\mathbb{R}^2)\times H^1(\mathbb{R}^2)$,
\begin{align}\label{GN:2}
    \int\frac{1}{36}|u|^4+\frac{9}{4}|v|^4+|u|^2|v|^2\leq C_{opt}^{(2)} K(u,v)M_{3\gamma}(u,v),
\end{align}
where
\begin{align*}
    C_{opt}^{(2)}= \frac{1}{2}\frac{1}{M_{3\gamma}(Q^*,P^*)}.
\end{align*}
Here $(Q^*,P^*)$ is the positive ground state of the following system
\begin{align}\label{equ:elliptic2}
    \begin{cases}
    -\Delta u+u=\frac{1}{9}u^3+2v^2u,\\
    -\Delta v+3\gamma v=9v^3+2u^2v.
    \end{cases}
\end{align}
Obviously, from \eqref{GN:1} and \eqref{GN:2}, we have $C_{opt}^{(1)}>C_{opt}^{(2)}$, which means that
\begin{align}\notag
    M_{3\gamma}(Q,P)<M_{3\gamma}(Q^*,P^*).
\end{align}

Now we give the following global result.
\begin{theorem}\label{Thm2}
Assume $N=2$, $c\in\mathbb{R}^2$ and $\alpha=3\gamma$. Let $u_0,v_0\in H^1(\mathbb{R}^2)\backslash\{0\}$ and
\begin{align}\label{con:initial}
    M_{3\gamma}(u_0,v_0)<M_{3\gamma}(Q^*,P^*).
\end{align}
 Set
\begin{align}\label{initialdata}
    (u_{0,c},v_{0,c})=\left(e^{\frac{i}{2}c\cdot x}u_0, e^{\frac{i\gamma}{2}c\cdot x}v_0\right).
\end{align}
Then the following statements hold.

(i) If $\gamma<3$, then there exist $A_0, A_1>0$ such that $\|v_0\|_{L^2}^2<A_0$ and $|c|>A_1$, the $H^1$ solution of \eqref{equ:system} with the initial data $(u(0),v(0))=(u_{0,c},v_{0,c})$ exists globally in time.

(ii) If $\gamma>3$, then there exist $B_0, B_1>0$ such that $\|u_0\|_{L^2}^2<B_0$ and $|c|>B_1$, the $H^1$ solution of \eqref{equ:system} with the initial data $(u(0),v(0))=(u_{0,c},v_{0,c})$ exists globally in time.
\end{theorem}

{\bf Comments:}

1. Using mass, energy conservation and the Gagliardo-Nirenberg inequality, one can easily obtain the global existence with the initial data $(u_{0,c},v_{0,c})\in H^1(\mathbb{R})\times H^1(\mathbb{R})$ in one-dimensional case. In the three-dimensional case,  however, global existence is only achieved within the set $A_{\omega,c,\alpha-3\gamma}^+$, see \eqref{def:A} and Lemma \ref{lemma:global}, this is similar to the classical case.

2. According to \cite{OP2021AMP}, the global existence condition \eqref{condition:global} is sharp when $\alpha=3\gamma=9$. However, for cases where $\gamma \neq 3$ but $\alpha=3\gamma$, we propose a weaker condition for the existence of global solutions: specifically, \eqref{con:initial} combined with a restriction on either the $L^2$-norm of $u_{0,c}$ or $v_{0,c}$ in the initial data. This new criterion is less restrictive than \eqref{condition:global} due to the inequality $M_{3\gamma}(Q,P) < M_{3\gamma}(Q^*,P^*)$. This observation highlights that, when $\gamma \neq 3$, the momentum alteration in the initial data (via \eqref{initialdata}) fundamentally affects the global behavior of the solution. This phenomenon arises from the lack of Galilean invariance in the system.

3. This contrasts with the global existence result in \cite[Theorem 1.3]{FHI2022MA}, where only one of the $L^2$-norms of the initial data $u(0)$ or $v(0)$ is constrained. The distinction arises due to the fundamental difference in nonlinear interactions: \cite{FHI2022MA} studied a quadratic-interaction system, whereas our cubic-interaction framework requires the additional constraint \eqref{con:initial}.

In order to obtain this Theorem \ref{Thm2}, we need to study  the traveling solitary wave solution of system \eqref{equ:system} of  the form
\begin{align}\label{traveling}
    \left(u_{c,\omega}(t,x),v_{c,\omega}(t,x)\right)=\left(e^{i(\omega-1) t}\phi(x-ct),e^{3i(\omega-1) t}\psi(x-ct)\right).
\end{align}
In particular, when $c=0$, \eqref{traveling} is the form of standing wave solutions, then the system \eqref{equ:system} reads
\begin{equation}\label{equ:c=0}
    \begin{cases}
    -\Delta \phi+\omega \phi=\left(\frac{1}{9}\phi^2+2\psi^2\right)\phi+\frac{1}{3}\bar{\phi}^2\psi,\\
     -\Delta \psi+(3\gamma\omega -3\gamma+\alpha)\psi=\left(9\psi^2+2\phi^2\right)\psi+\frac{1}{9}\phi^3.
    \end{cases}
\end{equation}
The existence of the ground state was obtained by Oliveira and Pastor, see \cite{OP2021AMP}.
When $\gamma=3$, one can obtain traveling solitary wave solutions from the standing wave solutions through the Galilean transformation. However, when $\gamma\neq3$, such a construction does not work due to the lack of Galilean symmetry. Therefore, it is interesting to investigate the existence and nonexistence of traveling solitary wave solutions of \eqref{equ:system} when $\gamma\neq3$.

We should mention that over the past 20 years, the following nonlinear Schr\"odinger system has received more mathematical attention:
\begin{equation*}
\begin{cases}
     i\partial_tu+\Delta u+\left(a_1|u|^2+b|v|^2\right)u=0,\\
     i\partial_tv+\Delta v+(a_2|v|^2+b|u|^2)v=0,
\end{cases}
\end{equation*}
This system has good variational framework,  and there are many results focusing on standing waves,  see\cite{WW2008ARMA,SW2013poincare,PPW2016CVPDE,LW2005CMP,LW2008CMP,CZ2012ARMA} and references therein. However, very few results focus on system \eqref{equ:system},  and the nonlinear term within it presents significant analytical challenges.

 For $(\omega,c)\in\mathbb{R}\times\mathbb{R}^N$ and $c\neq0$, \eqref{traveling} is a solution of \eqref{equ:system} if and only if
$\left(\phi(x-ct),\psi(x-ct)\right)$ is a solution of the stationary system
\begin{equation}\label{equ:s}
    \begin{cases}
    -\Delta\phi+\omega \phi+ic\cdot\nabla\phi-\left(\frac{1}{9}|\phi|^2+2|\psi|^2\right)\phi-\frac{1}{3}\bar{\phi}^2\psi=0,\\
-\Delta\psi+(3\gamma\omega-3\gamma+\alpha)\psi+i\gamma c\cdot\nabla\psi-\left(9|\psi|^2+2|\phi|^2\right)\psi-\frac{1}{9}\phi^3=0.
    \end{cases}
\end{equation}

For $(\phi,\psi)$ satisfying \eqref{equ:s}, let $ (\phi,\psi)=\left(e^{\frac{i}{2}c\cdot x}\Tilde\phi,e^{\frac{i\gamma}{2}c\cdot x}\Tilde\psi\right)$.
Then $(\Tilde\phi,\Tilde\psi)$ satisfies
\begin{align}\label{sys:complex}
    \begin{cases}
    -\Delta\Tilde{\phi}+\left(\omega-\frac{|c|^2}{4}\right)\Tilde{\phi}-\left(\frac{1}{9}|\Tilde{\phi}|^2+2|\Tilde{\psi}|^2\right)\Tilde{\phi}-\frac{1}{3}e^{i\left(\frac{\gamma}{2}-\frac{3}{2}\right)c\cdot x}\bar{\Tilde{\phi}}^2\Tilde{\psi}=0,\\
    -\Delta\Tilde{\psi}+\left(3\gamma\omega-3\gamma+\alpha-\frac{\gamma^2|c|^2}{4}\right)\Tilde{\psi}-\left(9|\Tilde{\psi}|^2+2|\Tilde{\phi}|^2\right)\Tilde{\psi}-\frac{1}{9}e^{i\left(\frac{3}{2}-\frac{\gamma}{2}\right)c\cdot x}\Tilde{\phi}^3 =0.
    \end{cases}
\end{align}

Now we state our existence and non-existence results.
\begin{theorem}\label{Thm1}
Let $1\leq N\leq3$, $\gamma>0$, $\gamma\neq3$, and $\omega>\max\left\{\frac{|c|^2}{4},\frac{\gamma|c|^2}{12}-\frac{\alpha-3\gamma}{3\gamma}\right\}$.

(i) Then there exists a boosted ground state solution $(u, v)$ of \eqref{equ:s} with $v\neq  0$. That is, the system \eqref{equ:system} has a traveling solitary wave solution in the form of \eqref{traveling}.

(ii) If $3\gamma\omega+\alpha-3\gamma>\omega{9}^{\frac{4}{4-N}}$ (in particular, if $\alpha=3\gamma$, this condition reduces to $3\gamma>9^{\frac{4}{4-N}}$), then the traveling solitary wave solutions obtained in (i) are non-trivial, that is, $u\neq  0$ and $v\neq  0$.
\end{theorem}




{\bf Comments:}

1. The existence of traveling solitary wave solutions for \eqref{equ:system} is equivalent to the existence of non-trivial solutions for the system \eqref{sys:complex}. Notably, when $\gamma \neq 3$ (the non-mass resonance condition), such solutions are essentially non-radial and complex-valued. The absence of radial symmetry and real-valuedness necessitates the development of a novel method to address these challenges. Consequently, this leads to a new and non-trivial existence result.

2. Part (i) of Theorem \ref{Thm1} establishes that traveling solitary wave solutions may exhibit either semi-trivial or non-trivial behavior. Specifically, a semi-trivial solution $(0, \tilde{Q})$ exists, where $\tilde{Q}$ is a Galilean-transformed version of $Q$, and $Q$ is the positive ground state solution of
\begin{align}\label{equ:Q}
    -\Delta Q+(3\gamma\omega+\alpha-3\gamma) Q=9Q^3.
\end{align}
This is completely different from the case of $\chi^{(2)}$.  In optical materials with $\chi^{(2)}$, the system reads as
\begin{align}\label{equ:x2}
\begin{cases}
     i\partial_tu+\frac{1}{2m}\Delta u=\bar{u}v,\\
     i\partial_tv+\frac{1}{2M}\Delta v=u^2,
\end{cases}
   (t,x)\in\mathbb{R}\times\mathbb{R}^N,~~m,M>0.
\end{align}
Fukaya, Hayashi and Inui \cite{FHI2022MA} established the existence of traveling solitary wave solutions in the absence of mass resonance.  Crucially, this scenario eliminates semi-trivial solutions, so the analytical process is simpler. For more about the system with quadratic interaction like \eqref{equ:x2}, one can see \cite{DF2021ZAMP,HOT2013Poincare,NP2022CVPDE,GW2023CVPDE} and the reference therein.

3. We employ an energy comparison technique to exclude semi-trivial solutions. This idea has been used many times in the literature, 
notably in \cite{OP2021AMP}, where Oliveira and Pastor studied the existence of standing-wave solutions (with $c=0$) for Equation \eqref{equ:system} and investigated other related results.

\begin{remark} When $0<3\gamma\omega+\alpha-3\gamma<\omega{9}^{\frac{4}{4-N}}$ and $\gamma\neq3$, the traveling solitary wave solutions may be semi-trivial or non-trivial.  This case remains an interesting open problem.


\end{remark}

Next, we aim to study the high frequency limit of the boosted ground state of \eqref{equ:system}.
\begin{theorem}\label{Thm:limit}
Assume $1\leq N\leq 3$. Let $(u_\omega,v_\omega)$ be the nontrivial traveling solitary wave solution of \eqref{equ:s} (see Theorem \ref{Thm1}). Assume that $|c|$ is bounded and $\omega\to+\infty$, then we have
\begin{align*}
\frac{1}{\omega^{\frac{1}{2}}}\left(u_{\omega}\left(\frac{x}{\omega^{\frac{1}{2}}}\right),v_{\omega}\left(\frac{x}{\omega^{\frac{1}{2}}}\right)\right)\to (u_\infty,v_\infty)~~\text{in}~~H^1(\mathbb{R}^N)\times H^1(\mathbb{R}^N).
\end{align*}
Here $(u_\infty,v_\infty)$ is the complex ground state solution of the system \eqref{equ:s} with $c=0$, $\omega=1$ and $\alpha=3\gamma$, that is
\begin{equation*}
    \begin{cases}
    -\Delta\phi+\phi-\left(\frac{1}{9}|\phi|^2+2|\psi|^2\right)\phi-\frac{1}{3}\bar{\phi}^2\psi=0,\\
-\Delta\psi+3\gamma\psi-\left(9|\psi|^2+2|\phi|^2\right)\psi-\frac{1}{9}\phi^3=0.
    \end{cases}
\end{equation*}
\end{theorem}

{\bf Comments:}

1. In Theorem \ref{Thm:limit}, we consider the high-frequency limit under the conditions that the velocity $|c|$ is bounded and the frequency $\omega\to+\infty$. In contrast, we can also examine the scenario where $|c|\to0$ while $\omega$ remians bounded.

2. Unlike in the single equation case (as studied in \cite{BCR2019PD,LM2022CVPDE}), the momentum term $P(u,v)$ is absent here. However, our analysis faces additional challenges due to the indeterminate sign of the term $c\cdot P(u,v)$ and the complex-valued nature of all functions.

\vspace{0.5cm}

This paper is organized as follows: In Section \ref{section:2}, we study the variational problem and show the existence and non-existence of traveling solitary wave solutions (Theorem \ref{Thm1}.  In Section \ref{section:4}, we establish the global existence (Theorem \ref{Thm2}). In Section \ref{section:3}, we establish the high frequency limit of the traveling solitary waves. The final section is an appendix.

\section{Existence of traveling waves}\label{section:2}
In this section, we aim to prove the existence of traveling solitary wave solutions by solving variational problems on the Nehari manifold. It is worth noting that we will be dealing with a complex elliptic system.

The energy functional with respect to \eqref{equ:s} is defined by
\begin{align}\label{def:functional}
    S_{\omega,c,\alpha-3\gamma}(u,v)=&E(u,v)+\frac{1}{2}(\omega-1)M_{3\gamma}(u,v)+\frac{1}{2}c\cdot P(u,v)\notag\\
    =&\frac{1}{2}K(u,v)+\frac{1}{2}\omega M_{3\gamma}(u,v)+\frac{1}{2}(\alpha-3\gamma)\int|v|^2+\frac{1}{2}c\cdot P(u,v)-D(u,v),
\end{align}
where  $E$, $K$, $M_{3\gamma}$ and $P$ are given by \eqref{energy}, \eqref{def:D}, \eqref{mass} and \eqref{momentum}, respectively.
For the sake of argument, we rewrite the energy functional \eqref{def:functional} as
\begin{align*}
    S_{\omega,c,\alpha-3\gamma}(u,v)=&\frac{1}{2}\int\left|\nabla\left(e^{-\frac{i}{2}c\cdot x}u\right)\right|^2+\frac{1}{2}\left(\omega-\frac{|c|^2}{4}\right)\|u\|_{L^2}^2\notag\\
    &+\frac{1}{2}\int\left|\nabla\left(e^{-\frac{i\gamma}{2}c\cdot x}v\right)\right|^2+\frac{1}{2}\left(3\gamma\omega-\frac{\gamma^2|c|^2}{4}\right)\|v\|_{L^2}^2+\frac{\alpha-3\gamma}{2}\|v\|_{L^2}^2-D(u,v).
\end{align*}
Now we define the function space $X_{\omega,c,\alpha-3\gamma}$ by
\begin{align*}
    X_{\omega,c,\alpha-3\gamma}=\left\{(u,v): \left(e^{-\frac{i}{2}c\cdot x}u,e^{-\frac{i\gamma}{2}c\cdot x}v\right)\in \Tilde{X}_{\omega,c,\alpha-3\gamma} \right\},
\end{align*}
where
\begin{align*}
\Tilde{X}_{\omega,c,\alpha-3\gamma}=
H^1(\mathbb{R}^N)\times H^1(\mathbb{R}^N),~\omega>\max\left\{\frac{|c|^2}{4},\frac{\gamma|c|^2}{12}-\frac{\alpha-3\gamma}{3\gamma}\right\},~\gamma>0,
\end{align*}
with the norm
\begin{align*}
\|(u,v)\|_{\Tilde{X}_{\omega,c,\alpha-3\gamma}}=\|\nabla u\|_{L^2}^2+\|\nabla v\|_{L^2}^2+\omega(\|u\|_{L^2}^2+3\gamma\|v\|_{L^2}^2)+(\alpha-3\gamma)\|v\|_{L^2}^2.
\end{align*}

Notice that $S_{\omega,c,\alpha-3\gamma}$ is defined on $X_{\omega,c,\alpha-3\gamma}$ if
\begin{align}\label{condition:dim}
1\leq N\leq3,~\omega>\max\left\{\frac{|c|^2}{4},\frac{\gamma|c|^2}{12}-\frac{\alpha-3\gamma}{3\gamma}\right\},~\gamma>0.
\end{align}
This condition comes from the Sobolev embedding, which is used to control the nonlinear term.  In this paper, we consider the case $\omega>\max\left\{\frac{|c|^2}{4},\frac{\gamma|c|^2}{12}-\frac{\alpha-3\gamma}{3\gamma}\right\}$ and $1\leq N\leq3$.

We denote the set of all non-zero solutions of \eqref{equ:s} by
\begin{align*}
    \mathcal{A}_{\omega,c,\alpha-3\gamma}=\{(\phi,\psi)\in X_{\omega,c,\alpha-3\gamma}~:~(\phi,\psi)\neq(0,0),~S^{\prime}_{\omega,c,\alpha-3\gamma}(\phi,\psi)=0\}
\end{align*}
and the set of all boosted ground states
\begin{align*}
    \mathcal{G}_{\omega,c,\alpha-3\gamma}=\{(\phi,\psi)\in\mathcal{A}_{\omega,c,\alpha-3\gamma},~S_{\omega,c,\alpha-3\gamma}(\phi,\psi)\leq S_{\omega,c,\alpha-3\gamma}(\phi_1,\psi_1),~\text{for all}~ (\phi_1,\psi_1)\in\mathcal{A}_{\omega,c,\alpha-3\gamma}\}.
\end{align*}
In particular, if $c=0$,  $\mathcal{G}_{\omega,0,\alpha-3\gamma}$ is the set of all ground states.

For the sake of simplicity, we use the following notation. We set
\begin{align}\label{def:Q}
\mathcal{Q}_{\omega,c,\alpha-3\gamma}(u,v)=&\frac{1}{2}\int|\nabla u|^2+\frac{1}{2}\int|\nabla v|^2+\frac{\omega}{2}(\|u\|_{L^2}^2+3\gamma\|v\|_{L^2}^2)+\frac{1}{2}(\alpha-3\gamma)\|v\|_{L^2}^2+\frac{1}{2}c\cdot P(u,v).
\end{align}
The energy functional \eqref{def:functional} reads
\[S_{\omega,c,\alpha-3\gamma}(u,v)=\mathcal{Q}_{\omega,c,\alpha-3\gamma}(u,v)-D(u,v),\]
where $D$ is defined in \eqref{def:D}.
Now we introduce the Nehari functional
\begin{align}\label{def:N}
    N_{\omega,c,\alpha-3\gamma}(u,v)=\partial_\tau S_{\omega,c,\alpha-3\gamma}(\tau u,\tau v)|_{\tau=1}=2\mathcal{Q}_{\omega,c,\alpha-3\gamma}(u,v)-4D(u,v).
\end{align}
The minimization problem is defined as
\begin{align}\label{min:a}
    \mu_{\omega,c,\alpha-3\gamma}=\inf\{S_{\omega,c,\alpha-3\gamma}(\phi,\psi)~:~(\phi,\psi)\in \mathcal{N}_{\omega,c,\alpha-3\gamma}\},
\end{align}
where
\begin{align*}
    \mathcal{N}_{\omega,c,\alpha-3\gamma}=\{(\phi,\psi)\in X_{\omega,c,\alpha-3\gamma}:~(\phi,\psi)\neq(0,0),~N_{\omega,c,\alpha-3\gamma}(\phi,\psi)=0\}.
\end{align*}
We define the minimizers $a_{\omega,c,\alpha-3\gamma}$ by
\begin{align*}
    a_{\omega,c,\alpha-3\gamma}=\{(\phi,\psi)\in \mathcal{N}_{\omega,c,\alpha-3\gamma}~:~S_{\omega,c,\alpha-3\gamma}(\phi,\psi)=\mu_{\omega,c,\alpha-3\gamma}\}.
\end{align*}
We also let
\begin{align*}
    \widetilde{\mathcal{Q}}_{\omega,c,\alpha-3\gamma}(u,v)
    =&\frac{1}{2}\int|\nabla u|^2+\frac{1}{2}\int|\nabla v|^2
    +\frac{1}{2}\left(\omega-\frac{|c|^2}{4}\right)\|u\|_{L^2}^2\\
    &+\frac{1}{2}\left(3\gamma\omega-\frac{\gamma^2|c|^2}{4}\right)\|v\|_{L^2}^2+\frac{\alpha-3\gamma}{2}\|v\|_{L^2}^2,\\
    \widetilde{D}(u,v)=&\int \left(\frac{1}{36}|u(x)|^4+\frac{9}{4}|v(x)|^4+|u(x)|^2|v(x)|^2+\frac{1}{9}\Re\left[ e^{\frac{-i(3-\gamma)}{2}c\cdot x}\left(\bar{u}^3(x)v(x)\right)\right]\right).
\end{align*}
If
$(\Tilde{u},\Tilde{v})=\left(e^{-\frac{i}{2}c\cdot x}u,e^{-\frac{i\gamma}{2}c\cdot x}v\right),$
then we have the following relations
\begin{align*}
    \Tilde{\mathcal{Q}}_{\omega,c,\alpha-3\gamma}(\Tilde{u},\Tilde{v})=\mathcal{Q}_{\omega,c,\alpha-3\gamma}(u,v),~~\widetilde{D}(\Tilde{u},\Tilde{v})=D(u,v).
\end{align*}
The corresponding  functionals are
\begin{align*}
    \Tilde{S}_{\omega,c,\alpha-3\gamma}(u,v)=&\Tilde{\mathcal{Q}}_{\omega,c,\alpha-3\gamma}(u,v)-\Tilde{D}(u,v),\\
    \Tilde{N}_{\omega,c,\alpha-3\gamma}(u,v)=&\partial_\tau\Tilde{S}_{\omega,c,\alpha-3\gamma}(\tau u,\tau v)|_{\tau=1}=2\Tilde{\mathcal{Q}}_{\omega,c,\alpha-3\gamma}(u,v)-4\Tilde{D}(u,v).
\end{align*}
We also use the following notations.
\begin{align*}
     &\widetilde{\mathcal{A}}_{\omega,c,\alpha-3\gamma}=\{(\phi,\psi)\in \Tilde{X}_{\omega,c,\alpha-3\gamma}:~(\phi,\psi)\neq(0,0),~\Tilde{S}^{\prime}_{\omega,c,\alpha-3\gamma}(\phi,\psi)=0\},\\
      &\widetilde{\mathcal{G}}_{\omega,c,\alpha-3\gamma}=\{(\phi,\psi)\in\widetilde{\mathcal{A}}_{\omega,c,\alpha-3\gamma}:~\Tilde{S}_{\omega,c,\alpha-3\gamma}(\phi,\psi)\leq \Tilde{S}_{\omega,c,\alpha-3\gamma}(\phi_1,\psi_1),~\text{for all}~ (\phi_1,\psi_1)\in\Tilde{\mathcal{A}}_{\omega,c,\alpha-3\gamma}\},\\
      &\Tilde{\mu}_{\omega,c,\alpha-3\gamma}=\inf\{\Tilde{S}_{\omega,c,\alpha-3\gamma}(\phi,\psi)~:~(\phi,\psi)\in \Tilde{\mathcal{N}}_{\omega,c,\alpha-3\gamma}\},\\
      &\Tilde{\mathcal{N}}_{\omega,c,\alpha-3\gamma}=\{(\phi,\psi)\in \Tilde{X}_{\omega,c,\alpha-3\gamma}:~(\phi,\psi)\neq(0,0),~\Tilde{N}_{\omega,c,\alpha-3\gamma}(\phi,\psi)=0\},\\
       &\Tilde{a}_{\omega,c,\alpha-3\gamma}=\{(\phi,\psi)\in \Tilde{\mathcal{N}}_{\omega,c,\alpha-3\gamma}:~\Tilde{S}_{\omega,c,\alpha-3\gamma}(\phi,\psi)=\Tilde{\mu}_{\omega,c,\alpha-3\gamma}\}.
\end{align*}
Notice that
\begin{align*}
    (\phi,\psi)\in {\mathcal{G}}_{\omega,c,\alpha-3\gamma}\Leftrightarrow (\Tilde{\phi},\Tilde{\psi})\in \Tilde{\mathcal{G}}_{\omega,c,\alpha-3\gamma},~~
     (\phi,\psi)\in a_{\omega,c,\alpha-3\gamma}\Leftrightarrow (\Tilde{\phi},\Tilde{\psi})\in \Tilde{a}_{\omega,c,\alpha-3\gamma},
\end{align*}
and
\begin{align*}
    \mu_{\omega,c,\alpha-3\gamma}=\Tilde{\mu}_{\omega,c,\alpha-3\gamma}.
\end{align*}
\begin{lemma}\label{Lemma:Pohozaev}
Let $(\phi,\psi)$ be the solution of \eqref{equ:s}, then we have the following identities:
\begin{align}\label{Pohozaev:1}
    \int\left(|\nabla \phi|^2+|\nabla \psi|^2\right)+\omega\int\left(|\phi|^2+3\gamma|\psi|^2\right)+(\alpha-3\gamma)\int|\psi|^2\notag\\
    +(ic\cdot\nabla \phi,\phi)+(i\gamma c\cdot\nabla \psi,\psi)
    =4D(\phi,\psi),
\end{align}
and
\begin{align}\label{Pohozaev:2}
    &\frac{N-2}{2}\int\left(|\nabla \phi|^2+|\nabla \psi|^2\right)+\frac{N}{2}\int\left(\omega|\phi|^2+3\gamma\omega|\psi|^2\right)\notag\\
    &+\frac{N}{2}(\alpha-3\gamma)\int|\psi|^2+\frac{N-1}{2}\left[(\phi,ic\cdot\nabla \phi)
    +\gamma(\psi,ic\cdot\nabla \psi)\right]=ND(\phi,\psi).
\end{align}
\end{lemma}
We give the proof in Appendix A.

Next, we state the following existence result.
 \begin{proposition}\label{pro}
Let  $1\leq N\leq3$, $\gamma>0$ and $\gamma\neq3$, then
 $$\mathcal{G}_{\omega,c,\alpha-3\gamma}=a_{\omega,c,\alpha-3\gamma}\neq \emptyset.$$
 In fact, this is equivalent to
 $\widetilde{\mathcal{G}}_{\omega,c,\alpha-3\gamma}=\Tilde{a}_{\omega,c,\alpha-3\gamma}\neq \emptyset.$
 \end{proposition}
If Proposition \ref{pro} is true, then we can easily obtain the Theorem \ref{Thm1} (i). we can use some arguments that can be found, for example, in \cite{FHI2022MA} in the context of $\chi^{(2)}$. For the reader’s convenience, we provide the details of the adaptation to our case in the appendix B.

In order to prove Theorem \ref{Thm1} (ii), we will prove that the minimizer $(u,v)$ is a non-semitrivial solution ($u\neq0$ and $v\neq0$) of the system \eqref{equ:s}.  If $u=0$, then $(0,\Tilde{Q})$ is the semi-trivial solution of system \eqref{equ:s}, where $\Tilde{Q}$ is a Galilean transformation of $Q$, and $Q$ is the positive ground state of equation \eqref{equ:Q}. Since $Q$ is the ground state of \eqref{equ:Q}, then $S_{\omega,0,\alpha-3\gamma}(0,Q)\leq S_{\omega,c,\alpha-3\gamma}(0,\Tilde{Q})$. As below, we give the relation between $S_{\omega,c,\alpha-3\gamma}(u,v)$ and $S_{\omega,0,\alpha-3\gamma}(0,Q)$.
\begin{lemma}\label{lemma:trivial}
Assume that $3\gamma\omega+\alpha-3\gamma>\omega{9}^{\frac{4}{4-N}}$. Then the functions $(u,v)$ in the Nehari manifold $\mathcal{N}_{\omega,c,\alpha-3\gamma}$ satisfy
\begin{align*}
    S_{\omega,c,\alpha-3\gamma}(u,v)<S_{\omega,c,\alpha-3\gamma}(0,Q)=S_{\omega,0,\alpha-3\gamma}(0,Q),
\end{align*}
where $Q$ is the positive ground state of \eqref{equ:Q}. In particular, $(0,Q)$ is not a minimizer of $S_{\omega,c,\alpha-3\gamma}$ in the complex space and $(u,v)$ is a non-semitrivial solution ($u\neq0$ and $v\neq0$).
\end{lemma}

\begin{remark}
 Notice that if $v=0$, then from the system \eqref{equ:s}, we can obtain that $u=0$. Hence, if the system has the semi-trivial solution, it must be of the form $(0,v)$. In particular, $(0,Q)$ is the special semi-trivial solution of the system \eqref{equ:s}, where $Q$ is the unique positive ground state solution of the equation \eqref{equ:Q}. The energy $S_{\omega,c,\alpha-3\gamma}(0,Q)$ is the ground state of this case; if there exists $(u,v)$ such that $S_{\omega,c,\alpha-3\gamma}(u,v)<S_{\omega,c,\alpha-3\gamma}(0,Q)$, we can obtain that $(u,v)$ is a nontrivial solution.    
\end{remark}

\begin{proof}
It suffices to prove that there exist  $\beta_1,\beta_2\in\mathbb{R}$ and the complex function $h\in H^1$ such that $(\beta_1\beta_2h,\beta_1Q)\in \mathcal{N}_{\omega,c,\alpha-3\gamma}$ and $S_{\omega,c,\alpha-3\gamma}(\beta_1\beta_2h,\beta_1Q)<S_{\omega,c,\alpha-3\gamma}(0,Q)$. In fact, $(\beta_1\beta_2h,\beta_1Q)\in \mathcal{N}_{\omega,c,\alpha-3\gamma}$ if and only if $N_{\omega,c,\alpha-3\gamma}(\beta_1\beta_2h,\beta_1Q)=0$. Since
\begin{align*}
    N_{\omega,c,\alpha-3\gamma}(\beta_1\beta_2h,\beta_1Q)=2\mathcal{Q}_{\omega,c,\alpha-3\gamma}(\beta_1\beta_2h,\beta_1Q)-4D(\beta_1\beta_2h,\beta_1Q),
\end{align*}
then by taking $\beta_1\in\mathbb{R}$ satisfying
\begin{align}\label{def:alpha}
    \beta_1^2=\frac{\mathcal{Q}_{\omega,c,\alpha-3\gamma}(\beta_2h,Q)}{2D(\beta_2h,Q)},
\end{align}
we see that $N_{\omega,c,\alpha-3\gamma}(\beta_1\beta_2h,\beta_1Q)=0$. From this point, we take $\beta_1$ as in \eqref{def:alpha}.

Now, from the identity
\begin{align*}
    \mathcal{Q}_{\omega,c,\alpha-3\gamma}(\beta_1\beta_2h,\beta_1Q)=2D(\beta_1\beta_2h,\beta_1Q)
\end{align*}
and \eqref{def:alpha}, we have
\begin{align*}
S_{\omega,c,\alpha-3\gamma}(\beta_1\beta_2h,\beta_1Q)=&\mathcal{Q}_{\omega,c,\alpha-3\gamma}(\beta_1\beta_2h,\beta_1Q)-D(\beta_1\beta_2h,\beta_1Q)\\
=&\frac{1}{2}\mathcal{Q}_{\omega,c,\alpha-3\gamma}(\beta_1\beta_2h,\beta_1Q)\\
=&\frac{\beta_1^2}{2}\mathcal{Q}_{\omega,c,\alpha-3\gamma}(\beta_2h,Q)
=\frac{\mathcal{Q}_{\omega,c,\alpha-3\gamma}^2(\beta_2h,Q)}{4D(\beta_2h,Q)}.
\end{align*}
Notice that
\begin{align*}
    S_{\omega,c,\alpha-3\gamma}(0,Q)=&\frac{1}{2}\int|\nabla Q|^2+\frac{3\gamma\omega+\alpha-3\gamma}{2}\|Q\|_{L^2}^2+\frac{1}{2}c\cdot P(0,Q)-\frac{9}{4}\int|Q|^4\\
    =&\frac{1}{2}\int|\nabla Q|^2+\frac{3\gamma\omega+\alpha-3\gamma}{2}\|Q\|_{L^2}^2-\frac{9}{4}\int|Q|^4
    ={\frac{1}{2}\mathcal{Q}_{\omega,0,\alpha-3\gamma}(0,Q)},
\end{align*}
where we used the fact that $Q$ is a positive, real and radial symmetry function, $P(0,Q)=0$ and the identity
\[\int|\nabla Q|^2+(3\gamma\omega-3\gamma+\alpha)\int|Q|^2=9\int|Q|^4.\]
Thus $S_{\omega,c,\alpha-3\gamma}(\beta_1\beta_2h,\beta_1Q)<S_{\omega,0,\alpha-3\gamma}(0,Q)$ is equivalent to
\begin{align}\label{relation:energy}
    \mathcal{Q}_{\omega,c,\alpha-3\gamma}^2(\beta_2h,Q)<{2}D(\beta_2h,Q)\mathcal{Q}_{\omega,0,\alpha-3\gamma}(0,Q).
\end{align}
Both sides of \eqref{relation:energy} are polynomials of degree $4$ in $\beta_2$. The leading coefficient of the polynomial in the left-hand side is
\begin{align*}
    \left(\frac{1}{2}\int|\nabla h|^2+\frac{\omega}{2}|h|^2+\frac{1}{2}c\cdot P(h,0)\right)^2
\end{align*}
whereas the leading coefficient of the polynomial in the right-hand side is
\begin{align*}
    {\frac{1}{18}}\int|h|^4\left(\frac{1}{2}\int|\nabla Q|^2+\frac{3\gamma\omega+\alpha-3\gamma}{2}|Q|^2\right).
\end{align*}
Therefore, \eqref{relation:energy} holds for $\beta_2$ sufficiently large, provided that
\begin{align}\label{relation:b1}
    \mathcal{Q}_{\omega,c,\alpha-3\gamma}^2(h,0)<{\frac{1}{18}}\|h\|_{L^4}^4\mathcal{Q}_{\omega,0,\alpha-3\gamma}(0,Q).
\end{align}
So, it will show that \eqref{relation:b1} holds for some $h\in H^1$. For that, assume that $h(x)=Q(\lambda x)$ for some $\lambda\in\mathbb{R}$ to be determined. With this definition, \eqref{relation:b1} is equivalent to
\begin{align}\label{relation:b2}
    \left(\frac{\lambda^{2}}{2}\int|\nabla Q|^2+\frac{\omega}{2}\int|Q|^2\right)^2<\frac{\lambda^{N}}{18}\mathcal{Q}_{\omega,0,\alpha-3\gamma}(0,Q)\|Q\|_{L^4}^4.
\end{align}
From the Pohozaev identities \eqref{Pohozaev:1} and \eqref{Pohozaev:2}, we have
\begin{align*}
    \int|\nabla Q|^2=\frac{(3\gamma\omega+\alpha-3\gamma) N}{4-N}\int|Q|^2=\frac{9N}{4}\int|Q|^4.
\end{align*}
Injecting this into \eqref{relation:b2}, we obtain
\begin{align}\label{relation:b3}
    &\left(\frac{\lambda^2}{2}\frac{(3\gamma\omega+\alpha-3\gamma) N}{4-N}\int|Q|^2+\frac{\omega}{2}\int|Q|^2\right)^2\notag\\
    <&\frac{\lambda^N}{18}\left(\frac{(3\gamma\omega+\alpha-3\gamma) N}{2(4-N)}\int|Q|^2+\frac{(3\gamma\omega+\alpha-3\gamma)}{2}\int|Q|^2\right)\frac{4(3\gamma\omega+\alpha-3\gamma)}{9(4-N)}\int|Q|^2.
\end{align}
That is
\begin{align*}
\left(\lambda^2(3\gamma\omega+\alpha-3\gamma) N+(4-N)\omega\right)^2<\frac{16\lambda^N}{81}(3\gamma\omega+\alpha-3\gamma)^2.
\end{align*}
Let
\begin{align*}
    f(\lambda)=\left(\lambda^2(3\gamma\omega+\alpha-3\gamma) N+(4-N)\omega\right)-\frac{4\lambda^{\frac{N}{2}}}{9}(3\gamma\omega+\alpha-3\gamma).
\end{align*}
It is easy to see that $f$ has a global minimum at the point $\lambda_0={9}^{\frac{2}{N-4}}$. In addition, $f(\lambda_0)=(4-N)(\omega-(3\gamma\omega+\alpha-3\gamma)\lambda_0^2)$. By assumption, $f(\lambda_0)<0$, which means that $\omega-(3\gamma\omega+\alpha-3\gamma)\lambda_0^2<0$. If we choose that $3\gamma\omega+\alpha-3\gamma>\omega{9}^{\frac{4}{4-N}}$, then we see that \eqref{relation:b3} holds for $\lambda=\lambda_0$ and the proof of this Lemma is complete.
\end{proof}

\begin{proof}[\bf Proof of Theorem \ref{Thm1}(ii).]
From Lemma \ref{lemma:trivial}, one can obtain the Theorem \ref{Thm1} (ii).
\end{proof}

\section{Global existence}\label{section:4}
In this section, we aim to obtain the global existence.  First, we introduce a subset of the energy space
\begin{align}\label{def:A}
    A^{+}_{\omega,c,\alpha-3\gamma}=\{(u,v)\in H^1(\mathbb{R}^N)\times H^1(\mathbb{R}^N)~:~S_{\omega,c,\alpha-3\gamma}(u,v)\leq \mu_{\omega,c,\alpha-3\gamma},~N_{\omega,c,\alpha-3\gamma}(u,v)\geq0\}.
\end{align}
Obviously, $A^+_{\omega,c,\alpha-3\gamma}$ is nonempty.

Now we show that $A^{+}_{\omega,c,\alpha-3\gamma}$ is an invariant set under the flow.
\begin{lemma}\label{lemma:invariant}
Assume \eqref{condition:dim}. Then the set $A^{+}_{\omega,c,\alpha-3\gamma}$ is invariant under the flow of \eqref{equ:system}.
\end{lemma}
\begin{proof}
Let $(u_0,v_0)\in A^+_{\omega,c,\alpha-3\gamma}$. It is obvious that $S_{\omega,c,\alpha-3\gamma}(u(t),v(t))\leq \mu_{\omega,c,\alpha-3\gamma}$ for all $t\in I$, where $I$ is the maximal existence interval of $H^1$-solution, since the corresponding  mass, energy and momentum are conserved.

Now we show that $N_{\omega,c,\alpha-3\gamma}(u(t),v(t))\geq0$ for all $t\in I$. If not, there exist $t_1,t_2\in I$ such that $N_{\omega,c,\alpha-3\gamma}(u(t_1),v(t_1))<0$ and $N_{\omega,c,\alpha-3\gamma}(u(t_2),v(t_2))=0$. By the uniqueness of Cauchy problem for \eqref{equ:system}, we have $(u(t_2),v(t_2))\neq(0,0)$. Moreover, from $S_{\omega,c,\alpha-3\gamma}(u(t),v(t))\leq \mu_{\omega,c,\alpha-3\gamma}$, we obtain $(u(t_2),v(t_2))\in a_{\omega,c,\alpha-3\gamma}\subset\mathcal{G}_{\omega,c,\alpha-3\gamma}$. This yields that
\begin{align*}
    (u(t),v(t))
    =\left(e^{i\omega (t-t_2)}u(t_2,x-c(t-t_2)),e^{3i\omega (t-t_2)}v(t_2,x-c(t-t_2))\right)
\end{align*}
for all $t\in\mathbb{R}$. In particular, $N_{\omega,c,\alpha-3\gamma}(u(t),v(t))=0$ for all $t\in \mathbb{R}$, which contradicts $N_{\omega,c,\alpha-3\gamma}(u(t_1),v(t_1))<0$.

By similar argument as above, we can prove the case $A^-_{\omega,c,\alpha-3\gamma}$. Now we complete the proof of this lemma.
\end{proof}

Now we give the following scaling property.
\begin{lemma}\label{lemma:scaling}
Let $(\omega,c)\in \mathbb{R}^+\times\mathbb{R}^N\backslash\{0\}$. Then
\begin{align*}
    \mu_{\omega,c,\alpha-3\gamma}=|\sigma|^{4-N}\mu_{\frac{\omega}{\sigma^2},\frac{c}{\sigma},\frac{\alpha-3\gamma}{\sigma^2}},~~\sigma>0.
\end{align*}
\end{lemma}
\begin{proof}
Let $\{(u_n,v_n)\}$ be a minimizing sequence for $\mu_{\omega,c,\alpha-3\gamma}$, that is,
\begin{align*}
    N_{\omega,c,\alpha-3\gamma}(u_n,v_n)=0,~~S_{\omega,c,\alpha-3\gamma}\to\mu_{\omega,c,\alpha-3\gamma},
\end{align*}
where $N_{\omega,c,\alpha-3\gamma}$ and $S_{\omega,c,\alpha-3\gamma}$ are defined in \eqref{def:N} and \eqref{def:functional}, respectively.

Let $(\Tilde{u}_n,\Tilde{v}_n)=\left(\frac{1}{\sigma}u_n\left(\frac{x}{\sigma}\right),\frac{1}{\sigma}v_n\left(\frac{x}{\sigma}\right)\right)$, $\sigma>0$. Then
\begin{align*}
    N_{\omega,c,\alpha-3\gamma}(u_n,v_n)=&2\mathcal{Q}_{\omega,c,\alpha-3\gamma}(u_n,v_n)-4D(u_n,v_n)\\
    =&2\Big(\frac{\sigma^{4-N}}{2}\int|\nabla\Tilde{u}_n|^2+\frac{\sigma^{4-N}}{2}\int|\nabla\Tilde{v}_n|^2+\frac{\omega\sigma^{2-N}}{2}\int|\Tilde{u}_n|^2\\
    &+\frac{(3\gamma\omega+\alpha-3\gamma)\sigma^{2-N}}{2}\int|\Tilde{v}_n|^2+\frac{\sigma^{3-N}}{2}c\cdot P(\Tilde{u}_n,\Tilde{v}_n)\Big)-4\sigma^{4-N}D(\Tilde{u}_n,\Tilde{v}_n)\\
    =&2\sigma^{4-N}\Big(\frac{1}{2}\int|\nabla\Tilde{u}_n|^2+\frac{1}{2}\int|\nabla\Tilde{v}_n|^2+\frac{\omega\sigma^{-2}}{2}\int|\Tilde{u}_n|^2\\
    &+\frac{(3\gamma\omega+\alpha-3\gamma)\sigma^{-2}}{2}\int|\Tilde{v}_n|^2+\frac{\sigma^{-1}}{2}c\cdot P(\Tilde{u}_n,\Tilde{v}_n)\Big)-4\sigma^{4-N}D(\Tilde{u}_n,\Tilde{v}_n)\\
    =&\sigma^{4-N}N_{\frac{\omega}{\sigma^2},\frac{c}{\sigma},\frac{\alpha-3\gamma}{\sigma^2}}(\Tilde{u}_n,\Tilde{v}_n),
\end{align*}
and
\begin{align*}
    S_{\omega,c,\alpha-3\gamma}(u_n,v_n)=&\frac{\sigma^{4-N}}{2}\int|\nabla\Tilde{u}_n|^2+\frac{\sigma^{4-N}}{2}\int|\nabla\Tilde{v}_n|^2+\frac{\omega\sigma^{2-N}}{2}\int|\Tilde{u}_n|^2\\
    &+\frac{(3\gamma\omega+\alpha-3\gamma)\sigma^{2-N}}{2}\int|\Tilde{v}_n|^2+\frac{\sigma^{3-N}}{2}c\cdot P(\Tilde{u},\Tilde{v})-\sigma^{4-N}D(\Tilde{u}_n,\Tilde{v}_n)\\
    =&\sigma^{4-N}S_{\frac{\omega}{\sigma^2},\frac{c}{\sigma},\frac{\alpha-3\gamma}{\sigma^2}}(\Tilde{u}_n,\Tilde{v}_n).
\end{align*}
Hence, we have
\begin{align*}
    &N_{\frac{\omega}{\sigma^2},\frac{c}{\sigma},\frac{\alpha-3\gamma}{\sigma^2}}(\Tilde{u}_n,\Tilde{v}_n)=|\sigma|^{N-4}N_{\omega,c,\alpha-3\gamma}(u_n,v_n)=0,\\
     &S_{\frac{\omega}{\sigma^2},\frac{c}{\sigma},\frac{\alpha-3\gamma}{\sigma^2}}(\Tilde{u}_n,\Tilde{v}_n)=|\sigma|^{N-4}S_{\omega,c,\alpha-3\gamma}(u_n,v_n)\to|\sigma|^{N-4}\mu_{\omega,c,\alpha-3\gamma}.
\end{align*}
This means that we can obtain the desired result.
\end{proof}

\begin{lemma}\label{lemma:global}
Assume the condition \eqref{condition:dim} holds. If the initial data $(u_0,v_0)\in A^+_{\omega,c,\alpha-3\gamma}$, then the $H^1$ solution $(u(t),v(t))$ of \eqref{equ:system} exists globally in time and
\begin{align*}
    \sup_{t\in\mathbb{R}}\|(u(t),v(t))\|_{H^1\times H^1}\leq C(\|(u_0,v_0)\|_{H^1\times H^1})<\infty.
\end{align*}
\end{lemma}
\begin{proof}
Combining Lemma \ref{lemma:invariant} and the definitions of $S_{\omega,c,\alpha-3\gamma}$, $\mathcal{Q}_{\omega,c,\alpha-3\gamma}$ and $N_{\omega,c,\alpha-3\gamma}$, we have
\begin{align*}
    \mu_{\omega,c,\alpha-3\gamma}
    \geq&S_{\omega,c,\alpha-3\gamma}(u(t),v(t))\\
    =&\frac{1}{2}\mathcal{Q}_{\omega,c,\alpha-3\gamma}(u(t),v(t))+\frac{1}{4}N_{\omega,c,\alpha-3\gamma}(u(t),v(t))\\
    \geq& \frac{1}{2}\mathcal{Q}_{\omega,c,\alpha-3\gamma}(u(t),v(t))\\
    \geq&\frac{1}{4}\left\|\nabla u(t)-\frac{i}{2}cu(t)\right\|_{L^2}^2+\frac{1}{4}\left\|\nabla v(t)-\frac{i\gamma}{2}cv(t)\right\|_{L^2}^2.
\end{align*}
By the conservation of  mass, we deduce
\begin{align*}
    &\|\nabla u(t)\|_{L^2}^2+\|\nabla v(t)\|_{L^2}^2\\
    \leq&C\left(\left\|\nabla u(t)-\frac{i}{2}cu(t)\right\|_{L^2}^2+\left\|\nabla v(t)-\frac{i\gamma}{2}cv(t)\right\|_{L^2}^2\right)+C\left(\|u(t)\|_{L^2}^2+\|v(t)\|_{L^2}^2\right)\\
    \leq&C\mu_{\omega,c,\alpha-3\gamma}+CM_{3\gamma}(u_0,v_0).
\end{align*}
This means that the $H^1$ solution is globally in time.
Now we complete the proof of this lemma.
\end{proof}

{\begin{remark}\label{remark:invariant}
We note that $\mu_{\omega,c,\alpha-3\gamma}$ is independent on the angle of $c$. Indeed, let $c_1,c_2\in\mathbb{R}^N$ satisfy $|c_1|=|c_2|$. Then there exists an orthogonal matrix $R$ such that $Rc_2=c_1$. Let $(u_n,v_n)$ be a minimizing sequence for $\mu_{\omega,c_1,\alpha-3\gamma}$, i.e., $(u_n,v_n)\in \mathcal{N}_{\omega,c_1,\alpha-3\gamma}$ and $S_{\omega,c_1,\alpha-3\gamma}(u_n,v_n)\to \mu_{\omega,c_1,\alpha-3\gamma}$. Let $(w_n(x),z_n(x))=(u_n(Rx),v_n(Rx))$, then
\begin{align*}
        N_{\omega,c_2,\alpha-3\gamma}(w_n,z_n)=N_{\omega,c_1,\alpha-3\gamma}(u_n,v_n)=0,\\
        S_{\omega,c_2,\alpha-3\gamma}(w_n,z_n)=S_{\omega,c_1,\alpha-3\gamma}(u_n,v_n)\to\mu_{\omega,c_1,\alpha-3\gamma}.
\end{align*}
This implies that $\mu_{\omega,c_2,\alpha-3\gamma}\leq \mu_{\omega,c_1,\alpha-3\gamma}$. Similarly, the inverse inequality holds.
In particular, when $\alpha=3\gamma$, $\mu_{\omega,\frac{c}{|c|},0}$ is constant independent of $|c|$.
\end{remark}}

\begin{proof}[{\bf Proof of Theorem \ref{Thm2}}]
To prove (i).
Let $A_0$ be chosen later. If $(u_0,v_0)\in H^1(\mathbb{R}^2)\times H^1(\mathbb{R}^N)$ satisfies $\|v_0\|_{L^2}^2<A_0$ and $0<\epsilon\leq\frac{1}{|c|^2}$ , then we claim that
\begin{align}\label{global:1}
    &S_{\frac{|\sqrt{1+\epsilon}c|^2}{4},c,0}\left(e^{\frac{i}{2}c\cdot x}u_0,e^{\frac{i\gamma}{2}c\cdot x}v_0\right)<\mu_{\frac{|\sqrt{1+\epsilon}c|^2}{4},c,0},\\\label{global:2}
    &N_{\frac{|\sqrt{1+\epsilon}c|^2}{4},c,0}\left(e^{\frac{i}{2}c\cdot x}u_0,e^{\frac{i\gamma}{2}c\cdot x}v_0\right)\geq0.
\end{align}
for large $|c|$.  If this claim is true, then we can obtain
\begin{align*}
    \left(e^{\frac{i}{2}c\cdot x}u_0,e^{\frac{i\gamma}{2}c\cdot x}v_0\right)\in A^+_{\frac{|\sqrt{1+\epsilon}c|^2}{4},c,0}
\end{align*}
for $|c|$ large enough. Hence, by Lemma \ref{lemma:global}, we can obtain the desired result.

Now we prove the claim. Indeed, by the definition of $\mathcal{Q}_{\omega,c,0}$, we have
\begin{align*}
    &\mathcal{Q}_{\frac{|\sqrt{1+\epsilon}c|^2}{4},c,0}(u,v)\\
    =&\frac{1}{2}\int|\nabla u|^2+\frac{1}{2}\int|\nabla v|^2+\frac{(1+\epsilon)|c|^2}{8}\|u\|_{L^2}^2+\frac{3\gamma(1+\epsilon)|c|^2}{8}\|v\|_{L^2}^2+\frac{1}{2}c\cdot P(u,c)\\
    =&\frac{1}{2}\left\|\nabla \left(e^{-\frac{i}{2}c\cdot x}u\right)\right\|_{L^2}^2+\frac{1}{2}\left\|\nabla \left(e^{-\frac{i\gamma}{2}c\cdot x}v\right)\right\|_{L^2}^2\\
    &+\frac{\epsilon|c|^2}{8}\|u\|_{L^2}^2+\frac{3(1+\epsilon)\gamma|c|^2-\gamma^2|c|^2}{8}\|v\|_{L^2}^2.
\end{align*}
From Lemma \ref{lemma:scaling} and above, \eqref{global:1} is equivalent to
\begin{align*}
    &S_{\frac{|\sqrt{1+\epsilon}c|^2}{4},c,0}\left(e^{\frac{i}{2}c\cdot x}u_0,e^{\frac{i\gamma}{2}c\cdot x}v_0\right)\notag\\
    =&\mathcal{Q}_{\frac{|\sqrt{1+\epsilon}c|^2}{4},c,0}\left(e^{\frac{i}{2}c\cdot x}u_0,e^{\frac{i\gamma}{2}c\cdot x}v_0\right)-D\left(e^{\frac{i}{2}c\cdot x}u_0,e^{\frac{i\gamma}{2}c\cdot x}v_0\right)\notag\\
    =&\frac{1}{2}\left\|\nabla u_0\right\|_{L^2}^2+\frac{1}{2}\left\|\nabla v_0\right\|_{L^2}^2+\frac{\epsilon|c|^2}{8}\|u_0\|_{L^2}^2+\frac{3(1+\epsilon)\gamma|c|^2-\gamma^2|c|^2}{8}\|v_0\|_{L^2}^2\notag\\
    &-\int \frac{1}{36}|u_0|^4+\frac{9}{4}|v_0|^4+|u_0|^2|v_0|^2+\frac{1}{9}\Re\left(e^{\frac{i(\gamma-3)}{2}c\cdot x}\bar{u}_0^3(x)v_0(x)\right)
    \leq|c|^2\mu_{\frac{1+\epsilon}{4},\frac{c}{|c|},0}.
\end{align*}
That is,
\begin{align}\label{global:3}
    &\frac{1}{2}\left\|\nabla u_0\right\|_{L^2}^2+\frac{1}{2}\left\|\nabla v_0\right\|_{L^2}^2+\frac{\epsilon|c|^2}{8}\|u_0\|_{L^2}^2-\int \frac{1}{36}|u_0|^4+\frac{9}{4}|v_0|^4+|u_0|^2|v_0|^2+\frac{1}{9}\Re\left(e^{\frac{i(\gamma-3)}{2}c\cdot x}\bar{u}_0^3(x)v_0(x)\right)\notag\\
    \leq&|c|^2\left(\mu_{\frac{1+\epsilon}{4},\frac{c}{|c|},0}-\left(\frac{\gamma(3(1+\epsilon)-\gamma)}{8}\right)\|v_0\|_{L^2}^2\right).
\end{align}
By the Gagliardo-Nirenberg inequality, we have
\begin{align*}
    &\frac{1}{2}\left\|\nabla u_0\right\|_{L^2}^2+\frac{1}{2}\left\|\nabla v_0\right\|_{L^2}^2-\int \frac{1}{36}|u_0|^4+\frac{9}{4}|v_0|^4+|u_0|^2|v_0|^2\notag\\
    \geq& \frac{1}{2}\left\|\nabla u_0\right\|_{L^2}^2+\frac{1}{2}\left\|\nabla v_0\right\|_{L^2}^2-\frac{1}{2}\frac{M_{3\gamma}(u_0,v_0)}{M_{3\gamma}(Q^*,P^*)}(\left\|\nabla u_0\right\|_{L^2}^2+\left\|\nabla v_0\right\|_{L^2}^2)\notag\\
    =&\frac{1}{2}\left(1-\frac{M_{3\gamma}(u_0,v_0)}{M_{3\gamma}(Q^*,P^*)}\right)(\left\|\nabla u_0\right\|_{L^2}^2+\left\|\nabla v_0\right\|_{L^2}^2)>0,
\end{align*}
where in the last step we used the assumption $M_{3\gamma}(u_0,v_0)<M_{3\gamma}(Q^*,P^*)$, (see \eqref{con:initial}).

Let
\begin{align*}
    A_0=\frac{8}{\gamma(3(1+\epsilon)-\gamma)}\mu_{\frac{1+\epsilon}{4},\frac{c}{|c|},0}.
\end{align*}
From Lemma \ref{lemma:scaling} and Remark \ref{remark:invariant}, $\mu_{\frac{1+\epsilon}{4},\frac{c}{|c|},0}$ is independent on $c$.
Then, by $\gamma<3$ and  Lemma \ref{lemma:positive}, we get $A_0>0$. Let $\|v_0\|_{L^2}^2<A_0$, then
\begin{align*}
    &|c|^2\left(\mu_{\frac{1+\epsilon}{4},\frac{c}{|c|},0}-\left(\frac{\gamma(3(1+\epsilon)-\gamma)}{8}\right)\|v_0\|_{L^2}^2\right)\to0,~~\text{as}~~|c|\to\infty.
\end{align*}

On the other hand, since $0<\gamma<3$, the Riemann-Lebesgue theorem implies that
\begin{align}\label{global:4}
   \frac{1}{9}\int\Re\left(e^{\frac{i(\gamma-3)}{2}c\cdot x}\bar{u}_0^3(x)v_0(x)\right)\to 0~~\text{as}~~|c|\to\infty.
\end{align}
From \eqref{global:4}, we deduce that the left-hand side of \eqref{global:3} is bounded as $|c|$ is sufficiently large, since $0<\epsilon\leq \frac{1}{|c|^2}$. Therefore, \eqref{global:3} holds for $|c|$ large enough, and so does \eqref{global:1}.

Again, we obtain
\begin{align*}
    &N_{\frac{|\sqrt{1+\epsilon}c|^2}{4},c,0}\left(e^{\frac{i}{2}c\cdot x}u_0,e^{\frac{i\gamma}{2}c\cdot x}v_0\right)\\
    =&2\mathcal{Q}_{\frac{|\sqrt{1+\epsilon}c|^2}{4},c,0}\left(e^{\frac{i}{2}c\cdot x}u_0,e^{\frac{i\gamma}{2}c\cdot x}v_0\right)-4D\left(e^{\frac{i}{2}c\cdot x}u_0,e^{\frac{i\gamma}{2}c\cdot x}v_0\right)\\
    =&\left\|\nabla u_0\right\|_{L^2}^2+\left\|\nabla v_0\right\|_{L^2}^2+\frac{\epsilon|c|^2}{4}\|u_0\|_{L^2}^2+\frac{3(1+\epsilon)\gamma|c|^2-\gamma^2|c|^2}{4}\|v_0\|_{L^2}^2\\
    &-4\left(\int \frac{1}{36}|u_0|^4+\frac{9}{4}|v_0|^4+|u_0|^2|v_0|^2+\frac{1}{9}\Re\left(e^{\frac{i(\gamma-3)}{2}c\cdot x}\bar{u}_0^3(x)v_0(x)\right)\right).
\end{align*}
Since $u_0,v_0\in H^1$, then there exists a constant $C_0>0$ such that $\|u_0\|_{L^4}, \|v_0\|_{L^4}\leq C_0 $. Hence, for large $|c|$ and $0<\gamma<3$,
\begin{align*}
   \frac{\gamma|c|^2(3(1+\epsilon)-\gamma)}{4}\|v_0\|_{L^2}^2-4\int \frac{1}{36}|u_0|^4+\frac{9}{4}|v_0|^4+|u_0|^2|v_0|^2\geq0
\end{align*}
holds.
Combining the above estimates, we can obtain that
\begin{align*}
    N_{\frac{|\sqrt{1+\epsilon}c|^2}{4},c,0}\left(e^{\frac{i}{2}c\cdot x}u_0,e^{\frac{i\gamma}{2}c\cdot x}v_0\right)\geq0
\end{align*}
for $|c|$ large enough. Then, \eqref{global:2} holds. Hence, (i) holds.

To prove (ii).
Let $0<\epsilon\leq \frac{1}{|c|^2}$ and
$
   B_0=\frac{24}{(1+\epsilon)\gamma-3}\mu_{\frac{(1+\epsilon)\gamma}{12},\frac{c}{|c|},0}$.
If the initial data $(u_0,v_0)\in H^1(\mathbb{R}^2)\times H^1(\mathbb{R}^2)\backslash\{(0,0)\}$ satisfies $\|u_0\|_{L^2}^2<B_0$, then we can obtain, for $|c|$ large enough,
\begin{align*}
     &S_{\frac{(1+\epsilon)\gamma|c|^2}{12},c,0}\left(e^{\frac{i}{2}c\cdot x}u_0,e^{\frac{i\gamma}{2}c\cdot x}v_0\right)<\mu_{\frac{(1+\epsilon)\gamma|c|^2}{12},c,0},\\
    &N_{\frac{(1+\epsilon)\gamma|c|^2}{12},c,0}\left(e^{\frac{i}{2}c\cdot x}u_0,e^{\frac{i\gamma}{2}c\cdot x}v_0\right)\geq0.
\end{align*}
This means that, for $|c|$ large enough,
\begin{align*}
   \left(e^{\frac{i}{2}c\cdot x}u_0,e^{\frac{i\gamma}{2}c\cdot x}v_0\right)\in A^+_{\frac{(1+\epsilon)\gamma|c|^2}{12},c,0}.
\end{align*}
Hence, (ii) holds.
Now we complete the proof of Theorem \ref{Thm2}.
\end{proof}

\section{High frequency limit}\label{section:3}
In this section, we aim to study the high frequency limit of the boosted ground state of \eqref{equ:system} when the frequency is $\omega\to+\infty$.

\begin{lemma}\label{lemma:existence}\cite[Theorem 2.1]{OP2021AMP}
Let $1\leq N\leq3$, $\omega>0$. Then there exists at least one ground state $(u_{0,\omega},v_{0,\omega})$ to system \eqref{equ:c=0}, which is radially symmetric, $v_{0,\omega}$ is positive and $u_{0,\omega}$ is either positive or identically zero.
\end{lemma}

Let
\begin{align}\label{def:tilde:u}
    (\Tilde{u}_\omega,\Tilde{v}_\omega)=\frac{1}{\omega^{\frac{1}{2}}}\left(u\left(\frac{x}{\omega^{\frac{1}{2}}}\right),v\left(\frac{x}{\omega^{\frac{1}{2}}}\right)\right),~~\omega>0.
\end{align}
From Lemma \ref{lemma:scaling}, we have
\begin{align*}
    N_{1,\frac{c}{\sqrt{\omega}},\frac{\alpha-3\gamma}{\omega}}(\Tilde{u}_\omega,\Tilde{v}_\omega)
    =\frac{1}{\omega^{\frac{4-N}{2}}}N_{\omega,c,\alpha-3\gamma}(u,v),~~ S_{1,\frac{c}{\sqrt{\omega}},\frac{\alpha-3\gamma}{\omega}}(\Tilde{u}_\omega,\Tilde{v}_\omega)
    =\frac{1}{\omega^{\frac{4-N}{2}}}S_{\omega,c,\alpha-3\gamma}(u,v).
\end{align*}
This means that
$$\mu_{\omega,c,\alpha-3\gamma}=\omega^{\frac{4-N}{2}} \mu_{1,\frac{c}{\sqrt{\omega}},\frac{\alpha-3\gamma}{\omega}}.$$
So, $(u,v)$ is a minimizer of $\mu_{\omega,c,\alpha-3\gamma}$ is equivalent to the fact that $(\Tilde{u},\Tilde{v})$ is a minimizer of $\mu_{1,\frac{c}{\sqrt{\omega}},\frac{\alpha-3\gamma}{\omega}}$.

Now we give the following uniformly bounded lemma.
\begin{lemma}\label{lemma:u:bound}
Let $1\leq N\leq 3$. Then
$(\Tilde{u}_\omega,\Tilde{v}_\omega)$ is uniformly bounded in $H^1(\mathbb{R}^N)\times H^1(\mathbb{R}^N)$ as $|c|$ is bounded and $\omega\rightarrow +\infty$.
\end{lemma}
\begin{proof}
Let $(u_{0,\omega}, v_{0,\omega})$ be a radial minimizer of $\mu_{1,0,\frac{\alpha-3\gamma}{\omega}}$ given by Lemma \ref{lemma:existence}, then
$$\frac{c}{\sqrt{\omega}}\cdot P(u_{0,\omega}, v_{0,\omega})=0.$$
From this fact, we deduce that
$$N_{1,\frac{c}{\sqrt{\omega}},\frac{\alpha-3\gamma}{\omega}}(u_{0,\omega}, v_{0,\omega})=0.$$
This implies that
\begin{align}\label{relation:1a}
    \mu_{1,\frac{c}{\sqrt{\omega}},\frac{\alpha-3\gamma}{\omega}}\leq \mu_{1,0,\frac{\alpha-3\gamma}{\omega}}.
\end{align}
Next, we will consider the relation between $\mu_{1,0,\frac{\alpha-3\gamma}{\omega}}$ and $\mu_{1,0,0}$.

Let $(u_{0,0},v_{0,0})$ be a minimizer of $\mu_{1,0,0}$, where $\omega=1$ and $\alpha=3\gamma$ (see Lemma \ref{lemma:existence}).  From the  Pohozaev identity \eqref{Pohozaev:2},   $(u_{0,0},v_{0,0})$ satisfies
\begin{align}\label{inden:1a}
&\frac{N-2}{2}\int(|\nabla u_{0,0}|^2+|\nabla v_{0,0}|^2)+\frac{N}{2}\left(\|u_{0,0}\|_{L^2}^2+3\gamma\| v_{0,0}\|_{L^2}^2\right)
    =ND(u_{0,0},v_{0,0}).
\end{align}
Since $(u_{0,0},v_{0,0})$ is the solution of \eqref{equ:s} with $\omega=1$ and $\alpha=3\gamma$, then, from Lemma \ref{lemma:scaling}, we have
\begin{align}\label{inden:2a}
    N_{1,0,0}(u_{0,0},v_{0,0})=&2\mathcal{Q}_{1,0,0}(u_{0,0},v_{0,0})-4D(u_{0,0},v_{0,0})=0.
\end{align}
By the definition of $N_{1,0,0}$ (see \eqref{def:N}), \eqref{inden:2a} is equivalent to
\begin{align}\label{inden:3a}
    \int\left(|\nabla u_{0,0}|^2+|\nabla v_{0,0}|^2\right)&+\left(\|u_{0,0}\|_{L^2}^2+3\gamma\|v_{0,0}\|_{L^2}^2\right)-4D(u_{0,0},v_{0,0})=0.
\end{align}
Combining \eqref{inden:1a} and \eqref{inden:3a}, we get
\begin{align}\label{inden:4a}
   &(4-N)\int\left(|\nabla u_{0,0}|^2+|\nabla v_{0,0}|^2\right)
=N\left(\|u_{0,0}\|_{L^2}^2+3\gamma\|v_{0,0}\|_{L^2}^2\right).
\end{align}


Let
\begin{align*}
    \left(u_{0,0}^{\sigma}(x),v_{0,0}^{\sigma}(x)\right)=\left(u_{0,0}\left(\frac{x}{\sigma}\right),v_{0,0}\left(\frac{x}{\sigma}\right)\right),~~\sigma>0.
\end{align*}
Then
\begin{align*}
    N_{1,0,\frac{\alpha-3\gamma}{\omega}}\left(u_{0,0}^{\sigma},v_{0,0}^{\sigma}\right)=&\sigma^{N-2}\int\left(|\nabla u_{0,0}|^2+|\nabla  v_{0,0}|^2\right)+\sigma^N\left(\|u_{0,0}\|_{L^2}^2+3\gamma\|v_{0,0}\|_{L^2}^2\right)\\
    &+\sigma^N\frac{\alpha-3\gamma}{\omega}\|v_{0,0}\|_{L^2}^2-4\sigma^ND(u_{0,0},v_{0,0}),
\end{align*}
where $N_{1,0,\frac{\alpha-3\gamma}{\omega}}$ is defined in \eqref{def:N}.

From \eqref{inden:1a}, \eqref{inden:3a} and \eqref{inden:4a}, we have
\begin{align*}
    & N_{1,0,\frac{\alpha-3\gamma}{\omega}}\left(u_{0,0}^{\sigma},v_{0,0}^{\sigma}\right)\notag\\
    =&\sigma^{N-2}\int\left(|\nabla u_{0,0}|^2+|\nabla v_{0,0}|^2\right)-\frac{N\sigma^N}{4-N}\left(\|u_{0,0}\|_{L^2}^2+3\gamma\| v_{0,0}\|_{L^2}^2\right)+\frac{\sigma^N(\alpha-3\gamma)}{\omega}\|v_{0,0}\|_{L^2}^2\notag\\
    =&\frac{\sigma^{N-2}}{4-N}N\left(\|u_{0,0}\|_{L^2}^2+3\gamma\| v_{0,0}\|_{L^2}^2\right)-\frac{N\sigma^N}{4-N}(\left(\|u_{0,0}\|_{L^2}^2+3\gamma\| v_{0,0}\|_{L^2}^2-\frac{4-N}{N}\frac{\alpha-3\gamma}{\omega}\|v_{0,0}\|_{L^2}^2\right).
\end{align*}
If $\alpha-3\gamma<0$, then there exists $\sigma\in(0,1)$ such that $N_{1,0,\frac{\alpha-3\gamma}{\omega}}\left(u_{0,0}^{\sigma},v_{0,0}^{\sigma}\right)=0$. If $\alpha-3\gamma>0$, then $\frac{\alpha-3\gamma}{\omega}\|v_{0,0}\|_{L^2}^2\to0$ as $\omega\to\infty$. This implies that there exists $\sigma\in(1,2)$ such that $N_{1,0,\frac{\alpha-3\gamma}{\omega}}\left(u_{0,0}^{\sigma},v_{0,0}^{\sigma}\right)=0$ as $\omega\to\infty$. From the above argument, we can easily obtain that $\sigma\to1$ as $\omega\to\infty$. In particular,
if $\alpha-3\gamma=0$, obviously, $\sigma=1$ and $N_{1,0,\frac{\alpha-3\gamma}{\omega}}\left(u_{0,0}^{\sigma},v_{0,0}^{\sigma}\right)=0$.

Now, from the identity \eqref{inden:2a}, we deduce that
\begin{align*}
    &\sigma^{2-N}S_{1,0,\frac{\alpha-3\gamma}{\omega}}\left(u_{0,\omega}^{\sigma},v_{0,\omega}^{\sigma}\right)\\
    =&\frac{1}{2}\int\left(|\nabla u_{0,0}|^2+|\nabla v_{0,0}|^2\right)+\frac{\sigma^2}{2}\left(\|u_{0,0}\|_{L^2}^2+3\gamma\|u_{0,0}\|_{L^2}^2\right)+\frac{\sigma^2(\alpha-3\gamma)}{\omega}\|v_{0,0}\|_{L^2}^2-\sigma^2D(u_{0,0},v_{0,0})\notag\\
    =&\frac{1}{2}\int\left(|\nabla u_{0,0}|^2+|\nabla v_{0,0}|^2\right)+\sigma^2\frac{1}{4}\left(\|u_{0,0}\|_{L^2}^2+3\gamma\|v_{0,0}\|_{L^2}^2\right)\notag\\
    &-\frac{\sigma^2}{4}\int\left(|\nabla u_{0,0}|^2+|\nabla v_{0,0}|^2\right)+\frac{\sigma^2(\alpha-3\gamma)}{\omega}\|v_{0,0}\|_{L^2}^2\notag\\
    =&\left(\frac{1}{2}-\frac{\sigma^2}{4}\right)\int\left(|\nabla u_{0,0}|^2+|\nabla v_{0,0}|^2\right)+\sigma^2\frac{1}{4}\left(\|u_{0,0}\|_{L^2}^2+3\gamma\|v_{0,0}\|_{L^2}^2\right)+\frac{\sigma^2(\alpha-3\gamma)}{\omega}\|v_{0,0}\|_{L^2}^2\\
    =&\frac{1}{2}\mathcal{Q}_{1,0,0}(u_{0,0},v_{0,0})+\frac{1}{4}(1-\sigma^2)\int\left(|\nabla u_{0,0}|^2+|\nabla v_{0,0}|^2\right)+\frac{\sigma^2(\alpha-3\gamma)}{\omega}\|v_{0,0}\|_{L^2}^2\\
    &+\frac{1}{4}(\sigma^2-1)\left(\|u_{0,0}\|_{L^2}^2+3\gamma\|v_{0,0}\|_{L^2}^2\right)\notag\\
    =&\frac{1}{2}\mathcal{Q}_{1,0,0}(u_{0,0},v_{0,0})+\frac{\sigma^2(\alpha-3\gamma)}{\omega}\|v_{0,0}\|_{L^2}^2+o(1)\\
    =&\mu_{1,0,0}+\frac{\sigma^2(\alpha-3\gamma)}{\omega}\|v_{0,0}\|_{L^2}^2+o(1),
\end{align*}
where we used the fact that  $\mu_{1,0,0}=\frac{1}{2}\mathcal{Q}_{1,0,0}$ and  $\mathcal{Q}_{1,0,0}$ is defined by \eqref{def:Q}. This means that
\begin{align*}
    \sigma^{2-N}\mu_{1,0,\frac{\alpha-3\gamma}{\omega}}=\mu_{1,0,0}+\frac{\sigma^2(\alpha-3\gamma)}{\omega}\|v_{0,0}\|_{L^2}^2+o(1).
\end{align*}
Then we deduce that
\begin{align}\label{relation:2a}
    \mu_{1,0,\frac{\alpha-3\gamma}{\omega}}\leq \mu_{1,0,0}+C,~~\text{as}~~\omega\to\infty,
\end{align}
since $\sigma\to1$ as $\omega\to\infty$.
Combining \eqref{relation:1a} and \eqref{relation:2a}, we can obtain that $(\Tilde{u}_{\omega},\Tilde{v}_{\omega})$ is uniformly bounded in $H^1(\mathbb{R}^N)\times H^1(\mathbb{R}^N)$ as $|c|$ is bounded and $\omega\rightarrow +\infty$.

\end{proof}
In the next lemma we will give the relation between $\mu_{1,0,0}$ and $\mu_{1,\frac{c}{\sqrt{\omega}},\frac{\alpha-3\gamma}{\omega}}$.

\begin{lemma}\label{lemma:relation}
Let $1\leq N\leq 3$. Then $\mu_{1,0,0}-\mu_{1,\frac{c}{\sqrt{\omega}},\frac{\alpha-3\gamma}{\omega}}\rightarrow 0$ as $|c|$ is bounded and $\omega\rightarrow +\infty$.
\end{lemma}
\begin{proof}
From the  Pohozaev identities \eqref{Pohozaev:1} and \eqref{Pohozaev:2},   $(\Tilde{u}_\omega,\Tilde{v}_\omega)$ satisfies
\begin{align}\label{inden:1}
&\frac{N-2}{2}\int(|\nabla\Tilde{u}_w|^2+|\nabla\Tilde{v}_w|^2)+\frac{N}{2}\left(\|\Tilde{u}_\omega\|_{L^2}^2+3\gamma\| \Tilde{v}_\omega\|_{L^2}^2\right)+\frac{N}{2}\frac{\alpha-3\gamma}{\omega}\|\Tilde{v}_\omega\|_{L^2}^2\notag\\
&+\frac{N-1}{2}\frac{c}{\sqrt{\omega}}\cdot P(\Tilde{u}_\omega,\Tilde{v}_\omega)
    =ND(\Tilde{u}_\omega,\Tilde{v}_\omega).
\end{align}
Since $(\Tilde{u}_{\omega},\Tilde{v}_{\omega})$ is the solution of \eqref{equ:s}, then we have
\begin{align}\label{inden:2}
    N_{1,\frac{c}{\sqrt{\omega}},\frac{\alpha-3\gamma}{\omega}}(\Tilde{u}_{\omega},\Tilde{v}_{\omega})=2\mathcal{Q}_{1,\frac{c}{\sqrt{\omega}},\frac{\alpha-3\gamma}{\omega}}(\Tilde{u}_{\omega},\Tilde{v}_{\omega})-4D(\Tilde{u}_{\omega},\Tilde{v}_{\omega})=\omega^{\frac{N-4}{2}}N_{\omega,c,\alpha-3\gamma}(u,v)=0.
\end{align}
By the definition of $N_{1,\frac{c}{\sqrt{\omega}},\frac{\alpha-3\gamma}{\omega}}$ (see \eqref{def:N}), \eqref{inden:2} is equivalent to
\begin{align}\label{inden:3}
    \int\left(|\nabla\Tilde{u}_{\omega}|^2+|\nabla\Tilde{v}_{\omega}|^2\right)+\left(\|\Tilde{u}_{\omega}\|_{L^2}^2+3\gamma\|\Tilde{v}_{\omega}\|_{L^2}^2\right)+\frac{\alpha-3\gamma}{\omega}\|\Tilde{v}_\omega\|_{L^2}^2+\frac{c}{\sqrt{\omega}}\cdot P(\Tilde{u}_{\omega},\Tilde{v}_{\omega})-4D(\Tilde{u}_{\omega},\Tilde{v}_{\omega})=0.
\end{align}
Combining \eqref{inden:1} and \eqref{inden:3}, we get
\begin{align}\label{inden:4}
   &(4-N)\int\left(|\nabla \Tilde{u}_{\omega}|^2+|\nabla \Tilde{v}_{\omega}|^2\right)\notag\\
=&N\left(\|\Tilde{u}_{\omega}\|_{L^2}^2+3\gamma\|\Tilde{v}_{\omega}\|_{L^2}^2\right)+\frac{N(\alpha-3\gamma)}{\omega}\|\Tilde{v}_\omega\|_{L^2}^2+(N-2)\frac{c}{\sqrt{\omega}}\cdot P(\Tilde{u}_{\omega},\Tilde{v}_{\omega}).
\end{align}
Notice that $|c|$ is bounded, by H\"older inequality, Young inequality and Lemma \ref{lemma:u:bound}, we have, as $\omega\to\infty$,
\begin{align}\label{small}
    \frac{c}{\sqrt{\omega}}\cdot P(\Tilde{u}_{\omega},\Tilde{v}_\omega)\to0
\end{align}
and
\begin{align*}
    \frac{c}{\sqrt{\omega}}\cdot P(\Tilde{u}_\omega,\Tilde{v}_\omega)\leq\frac{1}{2}\left(\int|\nabla\Tilde{u}_\omega|^2+|\nabla\Tilde{v}_\omega|^2+|\Tilde{u}_\omega|^2+3\gamma|\Tilde{v}_\omega|^2\right).
\end{align*}
Hence, by the Gagliardo-Nirenberg inequality \eqref{GN:1}, we get
\begin{align*}
    &\frac{1}{2}\left(\int|\nabla\Tilde{u}_\omega|^2+|\nabla\Tilde{v}_\omega|^2+|\Tilde{u}_\omega|^2+3\gamma|\Tilde{v}_\omega|^2\right)\\
    <&2Q_{1,\frac{c}{\sqrt{\omega}},\frac{\alpha-3\gamma}{\omega}}(\Tilde{u}_\omega,\Tilde{v}_\omega)=4D(\Tilde{u}_\omega,\Tilde{v}_\omega)\\
    \leq& 4C_{opt}^{(1)}K(\Tilde{u}_\omega,\Tilde{v}_\omega)^{\frac{N}{2}}M_{3\gamma}(\Tilde{u}_\omega,\Tilde{v}_\omega)^{2-\frac{N}{2}}\\
\leq&4C_{opt}^{(1)}\left(\int|\nabla\Tilde{u}_\omega|^2+|\nabla\Tilde{v}_\omega|^2+|\Tilde{u}_\omega|^2+3\gamma|\Tilde{v}_\omega|^2\right)^2.
\end{align*}
This means that
\begin{align*}
    \left(\int|\nabla\Tilde{u}_\omega|^2+|\nabla\Tilde{v}_\omega|^2+|\Tilde{u}_\omega|^2+3\gamma|\Tilde{v}_\omega|^2\right)\gtrsim1,
\end{align*}
for $\omega$ is sufficiently large. Combining \eqref{inden:4} and \eqref{small}, then
\begin{align}\label{lower:bound}
\left(\int|\nabla\Tilde{u}_\omega|^2+|\nabla\Tilde{v}_\omega|^2\right)\gtrsim1,\ \  \left(\int|\Tilde{u}_\omega|^2+3\gamma|\Tilde{v}_\omega|^2\right)\gtrsim1.
\end{align}
Also, we have
\begin{align}\label{small:2}
    \frac{N(\alpha-3\gamma)}{\omega}\|\Tilde{v}_{\omega}\|_{L^2}^2\to0,~~\text{as}~~\omega\to\infty.
\end{align}
Let
\begin{align}\label{def:ut}
    \left(\Tilde{u}^\lambda_\omega(x),\Tilde{v}^\lambda_\omega(x)\right)=\left(\Tilde{u}_\omega(x/\lambda),\Tilde{v}_\omega(x/\lambda)\right),~~\lambda>0.
\end{align}
Then
\begin{align*}
    N_{1,0,0}\left(\Tilde{u}^\lambda_\omega,\Tilde{v}^\lambda_\omega\right)=&\lambda^{N-2}\int\left(|\nabla \Tilde{u}_\omega|^2+|\nabla \Tilde{v}_\omega|^2\right)+\lambda^N\left(\|\Tilde{u}_\omega\|_{L^2}^2+3\gamma\|\Tilde{v}_\omega\|_{L^2}^2\right)-4\lambda^ND(\Tilde{u}_\omega,\Tilde{v}_\omega),
\end{align*}
where $N_{1,0,0}$ is defined in \eqref{def:N}.

From \eqref{inden:1}, \eqref{inden:3} and \eqref{inden:4}, we have
\begin{align}\label{inden:5}
    &N_{1,0,0}\left(\Tilde{u}^\lambda_\omega,\Tilde{v}^\lambda_\omega\right)\notag\\
    =&\lambda^{N-2}\int\left(|\nabla \Tilde{u}_\omega|^2+|\nabla \Tilde{v}_\omega|^2\right)-\frac{N\lambda^N}{4-N}\left(\|\Tilde{u}_\omega\|_{L^2}^2+3\gamma\| \Tilde{v}_\omega\|_{L^2}^2+\frac{2c}{\sqrt{\omega}}\cdot P(\Tilde{u}_\omega,\Tilde{v}_\omega)\right)\notag\\
    =&\frac{\lambda^{N-2}}{4-N}\left(N\left(\|\Tilde{u}_\omega\|_{L^2}^2+3\gamma\| \Tilde{v}_\omega\|_{L^2}^2\right)+\frac{N(3-\gamma)}{\omega}\|\tilde{v}_{\omega}\|_{L^2}^2+(N-2)\frac{c}{\sqrt{\omega}}\cdot P(\Tilde{u}_\omega,\Tilde{v}_\omega)\right)\notag\\
    &-\frac{\lambda^N}{4-N}\left(N\left(\|\Tilde{u}_\omega\|_{L^2}^2+3\gamma\| \Tilde{v}_\omega\|_{L^2}^2\right)+\frac{4(3-\gamma)}{\omega}\|\tilde{v}_{\omega}\|_{L^2}^2+\frac{2c}{\sqrt{\omega}}\cdot P(\Tilde{u}_\omega,\Tilde{v}_\omega)\right).
\end{align}
Notice that
\begin{align*}
    N_{1,0,0}(\Tilde{u}_\omega,\Tilde{v}_\omega)=&N_{1,\frac{c}{\sqrt{\omega}},\frac{\alpha-3\gamma}{\omega}}(\Tilde{u}_\omega,\Tilde{v}_\omega)-\frac{\alpha-3\gamma}{\omega}\|\Tilde{v}_\omega\|_{L^2}^2-\frac{c}{\sqrt{\omega}}\cdot P(\Tilde{u}_\omega,\Tilde{v}_\omega)\\
    =&-\frac{\alpha-3\gamma}{\omega}\|\Tilde{v}_\omega\|_{L^2}^2-\frac{c}{\sqrt{\omega}}\cdot P(\Tilde{u}_\omega,\Tilde{v}_\omega).
\end{align*}

{\bf Case 1.} If $\frac{c}{\sqrt{\omega}}\cdot P(\Tilde{u}_\omega,\Tilde{v}_\omega)>0$ and $\alpha-3\gamma>0$, then $N_{1,0,\frac{\alpha-3\gamma}{\omega}}(\Tilde{u}_\omega,\Tilde{v}_\omega)<0$.  Obviously, using \eqref{lower:bound}, if $|c|$ is bounded and $\omega\to+\infty$,
\begin{align*}
    N\left(\|\Tilde{u}_\omega\|_{L^2}^2+3\gamma\| \Tilde{v}_\omega\|_{L^2}^2\right)+\frac{N(3-\gamma)}{\omega}\|\tilde{v}_{\omega}\|_{L^2}^2+(N-2)\frac{c}{\sqrt{\omega}}\cdot P(\Tilde{u}_\omega,\Tilde{v}_\omega)>0,\\
    N\left(\|\Tilde{u}_\omega\|_{L^2}^2+3\gamma\| \Tilde{v}_\omega\|_{L^2}^2\right)+\frac{4(3-\gamma)}{\omega}\|\tilde{v}_{\omega}\|_{L^2}^2+\frac{2c}{\sqrt{\omega}}\cdot P(\Tilde{u}_\omega,\Tilde{v}_\omega)>0,
\end{align*}
and
\begin{align*}
    (N-2)\frac{c}{\sqrt{\omega}}\cdot P(\Tilde{u}_\omega,\Tilde{v}_\omega)<\frac{2c}{\sqrt{\omega}}\cdot P(\Tilde{u}_\omega,\Tilde{v}_\omega),~~1\leq N\leq3.
\end{align*}
Hence, from \eqref{inden:5}, there exists $\lambda_\omega\in(0,1)$ such that
\begin{align*}
    N_{1,0,0}\left(\Tilde{u}^{\lambda_\omega}_\omega,\Tilde{v}^{\lambda_\omega}_\omega\right)=0.
\end{align*}

{\bf Case 2.} If $\frac{c}{\sqrt{\omega}}\cdot P(\Tilde{u}_\omega,\Tilde{v}_\omega)<0$ and $\alpha-3\gamma<0$, then $N_{1,0,0}(\Tilde{u}_\omega,\Tilde{v}_\omega)>0$. Similarly,
\begin{align*}
    N\left(\|\Tilde{u}_\omega\|_{L^2}^2+3\gamma\| \Tilde{v}_\omega\|_{L^2}^2\right)+\frac{N(3-\gamma)}{\omega}\|\tilde{v}_{\omega}\|_{L^2}^2+(N-2)\frac{c}{\sqrt{\omega}}\cdot P(\Tilde{u}_\omega,\Tilde{v}_\omega)>0,\\
    N\left(\|\Tilde{u}_\omega\|_{L^2}^2+3\gamma\| \Tilde{v}_\omega\|_{L^2}^2\right)+\frac{4(3-\gamma)}{\omega}\|\tilde{v}_{\omega}\|_{L^2}^2+\frac{2c}{\sqrt{\omega}}\cdot P(\Tilde{u}_\omega,\Tilde{v}_\omega)>0.
\end{align*}
Since $1\leq N\leq 3$, then
\begin{align*}
    (N-2)\frac{c}{\sqrt{\omega}}\cdot P(\Tilde{u}_\omega,\Tilde{v}_\omega)>\frac{2c}{\sqrt{\omega}}\cdot P(\Tilde{u}_\omega,\Tilde{v}_\omega).
\end{align*}
Hence, from \eqref{inden:5} there exists $\lambda_\omega\in(1,+\infty)$ such that
\begin{align*} N_{1,0,0}\left(\Tilde{u}^{\lambda_\omega}_\omega,\Tilde{v}^{\lambda_\omega}_\omega\right)=0.
\end{align*}
We claim that
\begin{align*}
    \lambda_\omega\in(1,+\infty) ~~\text{uniformly bounded}.
\end{align*}
In fact,  by assumption, $1\leq N\leq3$, $|c|$ is bounded and $\omega\to+\infty$, and the estimates \eqref{small} and \eqref{lower:bound}, we deduce that
\begin{align*}
    &\frac{1}{2}\left(N\left(\|\Tilde{u}_\omega\|_{L^2}^2+3\gamma\| \Tilde{v}_\omega\|_{L^2}^2\right)+\frac{N(3-\gamma)}{\omega}\|\tilde{v}_{\omega}\|_{L^2}^2+(N-2)\frac{c}{\sqrt{\omega}}\cdot P(\Tilde{u}_\omega,\Tilde{v}_\omega)\right)\\<&N\left(\|\Tilde{u}_\omega\|_{L^2}^2+3\gamma\| \Tilde{v}_\omega\|_{L^2}^2\right)+\frac{4(3-\gamma)}{\omega}\|\tilde{v}_{\omega}\|_{L^2}^2+\frac{2c}{\sqrt{\omega}}\cdot P(\Tilde{u}_\omega,\Tilde{v}_\omega).
\end{align*}
Then, from \eqref{inden:5}, we can obtain that $\lambda^2_{\omega}<2$. Hence, the claim is true.

{\bf Case 3.}  Assume that $\frac{c}{\sqrt{\omega}}\cdot P(\Tilde{u}_\omega,\Tilde{v}_\omega)>0$ and $\alpha-3\gamma<0$. If $N_{1,0,0}(\Tilde{u}_\omega,\Tilde{v}_\omega)<0$, then by the similar argument as Case 1, there exists $\lambda_\omega\in (0,1)$  such that $N_{1,0,0}(\Tilde{u}_{\omega}^{\lambda_\omega},\Tilde{v}_{\omega}^{\lambda_\omega})=0$; if  $N_{1,0,0}(\Tilde{u}_\omega,\Tilde{v}_\omega)>0$, then by the similar argument as Case 2, there exists $\lambda_\omega\in (1,\infty)$ is uniformly bounded,  such that $N_{1,0,0}(\Tilde{u}_{\omega}^{\lambda_\omega},\Tilde{v}_{\omega}^{\lambda_\omega})=0$.

{\bf Case 4.}  Assume that $\frac{c}{\sqrt{\omega}}\cdot P(\Tilde{u}_\omega,\Tilde{v}_\omega)>0$ and $\alpha-3\gamma>0$. By the same argument as Case 3, we can obtain the similar result. Here we omit it.

Furthermore, we claim that
\begin{align}\label{claim:3}
    \lambda_{\omega}\to1~~\text{as}~~\omega\to+\infty.
\end{align}
Indeed,  from \eqref{inden:5}, we have
\begin{align*}
0=N_{1,0,0}\left(\Tilde{u}^{\lambda_{\omega}}_\omega,\Tilde{v}^{\lambda_\omega}_\omega\right)
    =\frac{\lambda_{\omega}^{N-2}}{4-N}\Bigg(N(1-\lambda_\omega^2)\left(\|\Tilde{u}_\omega\|_{L^2}^2+3\gamma\| \Tilde{v}_\omega\|_{L^2}^2\right)+[N-2-2\lambda_\omega^2]\frac{c}{\sqrt{\omega}}\cdot P(\Tilde{u}_\omega,\Tilde{v}_\omega)\Bigg).
\end{align*}
By \eqref{small} and \eqref{lower:bound}, we can obtain the claim \eqref{claim:3}.

Now, from the identity \eqref{inden:2}, we deduce that
\begin{align*}
    &\lambda_{\omega}^{2-N}S_{1,0,0}\left(\Tilde{u}_{\omega}^{\lambda_\omega},\Tilde{v}_{\omega}^{\lambda_\omega}\right)\\
    =&\frac{1}{2}\int\left(|\nabla \Tilde{u}_{\omega}|^2+|\nabla \Tilde{v}_{\omega}|^2\right)+\frac{\lambda_\omega^2}{2}\left(\|\Tilde{u}_{\omega}\|_{L^2}^2+3\gamma\|\Tilde{v}_{\omega}\|_{L^2}^2\right)-\lambda_\omega^2D(\Tilde{u}_{\omega},\Tilde{v}_{\omega})\notag\\
    =&\frac{1}{2}\int\left(|\nabla \Tilde{u}_{\omega}|^2+|\nabla \Tilde{v}_{\omega}|^2\right)+\lambda_{\omega}^2\frac{1}{4}\left(\|\Tilde{u}_{\omega}\|_{L^2}^2+3\gamma\|\Tilde{v}_{\omega}\|_{L^2}^2\right)-\frac{\lambda_\omega^2(\alpha-3\gamma)}{4\omega}\|\Tilde{v}_\omega\|_{L^2}^2\\
    &-\frac{\lambda_{\omega}^2}{4}\int\left(|\nabla \Tilde{u}_{\omega}|^2+|\nabla \Tilde{v}_{\omega}|^2\right)-\lambda_\omega^2\frac{1}{4}\frac{c}{\sqrt{\omega}}\cdot P(\Tilde{u}_{\omega},\Tilde{v}_{\omega})\notag\\
    =&\left(\frac{1}{2}-\frac{\lambda_\omega^2}{4}\right)\int\left(|\nabla \Tilde{u}_{\omega}|^2+|\nabla \Tilde{v}_{\omega}|^2\right)+\lambda_{\omega}^2\frac{1}{4}\left(\|\Tilde{u}_{\omega}\|_{L^2}^2+3\gamma\|\Tilde{v}_{\omega}\|_{L^2}^2\right)\\
    &-\frac{\lambda_\omega^2(\alpha-3\gamma)}{4\omega}\|\Tilde{v}_\omega\|_{L^2}^2-\lambda_\omega^2\frac{1}{4}\frac{c}{\sqrt{\omega}}\cdot P(\Tilde{u}_{\omega},\Tilde{v}_{\omega})\notag\\
    =&\frac{1}{2}\mathcal{Q}_{1,\frac{c}{\sqrt{w}},\frac{\alpha-3\gamma}{\omega}}(\Tilde{u}_{\omega},\Tilde{v}_{\omega})+\frac{1}{4}(1-\lambda_\omega^2)\int\left(|\nabla \Tilde{u}_{\omega}|^2+|\nabla \Tilde{v}_{\omega}|^2\right)-\frac{(1+\lambda_\omega^2)(\alpha-3\gamma)}{4\omega}\|\Tilde{v}_\omega\|_{L^2}^2\\
    &+\frac{1}{4}(\lambda_\omega^2-1)\left(\|\Tilde{u}_{\omega}\|_{L^2}^2+3\gamma\|\Tilde{v}_{\omega}\|_{L^2}^2\right)-\frac{1}{4}(1+\lambda_\omega^2)\frac{c}{\sqrt{\omega}}\cdot P(\Tilde{u}_{\omega},\Tilde{v}_{\omega})\notag\\
    =&\mu_{1,\frac{c}{\sqrt{\omega}},\frac{\alpha-3\gamma}{\omega}}+o(1),~~\text{as}~~\omega\to+\infty.
\end{align*}
where $\mathcal{Q}_{1,\frac{c}{\sqrt{\omega}},\frac{\alpha-3\gamma}{\omega}}$ is defined by \eqref{def:Q} and in the last step,  we also used the Lemma  \ref{lemma:u:bound}, \eqref{claim:3} and \eqref{small:2}. This means that
\begin{align*}
    \mu_{1,0,0}= \mu_{1,\frac{c}{\sqrt{\omega}},\frac{\alpha-3\gamma}{\omega}}+o(1),~~\text{as}~~\omega\to+\infty.
\end{align*}
Hence, we can obtain $\mu_{1,0,0}-\mu_{1,\frac{c}{\sqrt{\omega}},\frac{\alpha-3\gamma}{\omega}}\rightarrow 0$ as $|c|$ is bounded and $\omega\rightarrow +\infty$. Now we have completed the proof of this Lemma.
\end{proof}

\begin{lemma}\label{lemma:con}
Up to a subsequence, $(\Tilde{u}_{\omega}^{\lambda_\omega},\Tilde{u}_{\omega}^{\lambda_\omega})\to(\Tilde{u}_\infty,\Tilde{v}_\infty)$ strongly in $H^1\times H^1$  as $\omega\rightarrow +\infty$. Moreover, $(\Tilde{u}_\infty,\Tilde{v}_\infty)$ is a minimizer of $\mu_{1,0,0}$.
\end{lemma}
\begin{proof}
From Lemma \ref{lemma:relation}, it is known that
\begin{align*}
    N_{1,0,0}(\Tilde{u}_{\omega}^{\lambda_\omega},\Tilde{u}_{\omega}^{\lambda_\omega})=0
\end{align*}
and
\begin{align*}
    S_{1,0,0}(\Tilde{u}_{\omega}^{\lambda_\omega},\Tilde{u}_{\omega}^{\lambda_\omega})\to \mu_{1,0,0}~~\text{as}~~\omega\to\infty.
\end{align*}
This means that $\left(\Tilde{u}_\omega^{\lambda_\omega},\Tilde{v}_\omega^{\lambda_\omega}\right)$ is the minimizing sequence of \eqref{min:a} with $c=0$ and $\alpha=3\gamma$.
By the definition of $(\Tilde{u}_{\omega}^{\lambda_\omega},\Tilde{u}_{\omega}^{\lambda_\omega})$ (see \eqref{def:ut}) and Lemma \ref{lemma:u:bound}, we get $(\Tilde{u}_{\omega}^{\lambda_\omega},\Tilde{u}_{\omega}^{\lambda_\omega})$ is also uniformly bounded in $H^1(\mathbb{R}^N)\times H^1(\mathbb{R}^N)$.

By similar argument as Lemma \ref{lemma:positive} and \ref{lemma:w}, there exist $\{y_\omega\}$, $(\Tilde{u}_\infty,\Tilde{v}_\infty)\in H^1(\mathbb{R}^N)\times H^1(\mathbb{R}^N)\backslash\{(0,0)\}$, and a subsequence $\tau_{y_\omega}\left(\Tilde{u}_\omega^{\lambda_\omega},\Tilde{v}_\omega^{\lambda_\omega}\right)$ (still denoted by $\tau_{y_\omega}\left(\Tilde{u}_\omega^{\lambda_\omega},\Tilde{v}_{\omega}^{\lambda_\omega}\right)$), where $\tau_{y_\omega}$ is defined as \eqref{def:tau},  such that
\begin{align*}
    \tau_{y_\omega}(\Tilde{u}_\omega^{\lambda_\omega},\Tilde{v}_\omega^{\lambda_\omega})\rightharpoonup (\tilde{u}_\infty,\tilde{v}_\infty)~~\text{weakly in}~~H^1(\mathbb{R}^N)\times H^1(\mathbb{R}^N).
\end{align*}
By weakly convergence of $\tau_{y_\omega}(\Tilde{u}_\omega^{\lambda_\omega},\Tilde{v}_{\omega}^{\lambda_\omega})$ and Lemma \ref{lemma:w}, we have
\begin{align*}
\mathcal{Q}_{1,0,0}\left(\tau_{y_\omega}(\Tilde{u}_\omega^{\lambda_\omega},\Tilde{v}_{\omega}^{\lambda_\omega})\right)-\mathcal{Q}_{1,0,0}\left(\tau_{y_\omega}(\Tilde{u}_\omega^{\lambda_\omega},\Tilde{v}_\omega^{\lambda_\omega})-(\Tilde{u}_\infty,\Tilde{v}_\infty)\right)\to \mathcal{Q}_{1,0,0}(\Tilde{u}_\infty,\Tilde{v}_\infty).
\end{align*}
Since $\mathcal{Q}_{1,0,0}(\Tilde{u}_\infty,\Tilde{v}_{\infty})>0$, up to a subsequence, we have
\begin{align*}
    \frac{1}{2}\lim_{\omega\to\infty}\mathcal{Q}_{1,0,0}\left(\tau_{y_\omega}(\Tilde{u}_\omega^{\lambda_\omega},\Tilde{v}_\omega^{\lambda_\omega})-(\Tilde{u}_\infty,\Tilde{v}_\infty)\right)<\frac{1}{2}\lim_{\omega\to\infty}\mathcal{Q}_{1,0,0}(\tau_{y_{\omega}}(\Tilde{u}_{\omega}^{\lambda_\omega},\Tilde{v}_\omega^{\lambda_\omega}))=\mu_{1,0,0}.
\end{align*}
By a similar argument as Lemma \ref{lemma:upper}, we can obtain that
\begin{align*}
    \mu_{1,0,0}\leq\frac{1}{2}\mathcal{Q}_{1,0,0}(\tilde{u}_\infty,\tilde{v}_\infty).
\end{align*}
Then, using the above and the weakly lower semi-continuity of norms, we obtain
\begin{align*}
    \mu_{1,0,0}\leq\frac{1}{2}\mathcal{Q}_{1,0,0}(\Tilde{u}_\infty,\Tilde{v}_\infty)\leq\frac{1}{2}\lim_{\omega\to\infty}\mathcal{Q}_{1,0,0}(\tau_{y_{\omega}}(\Tilde{u}_{\omega}^{\lambda_\omega},\Tilde{v}_\omega^{\lambda_\omega}))=\mu_{1,0,0}.
\end{align*}
Hence, we have
\begin{align*}
    \mathcal{Q}_{1,0,0}\left(\tau_{y_\omega}(\Tilde{u}_\omega^{\lambda_\omega},\Tilde{v}_\omega^{\lambda_\omega})-(\Tilde{u}_\infty,\Tilde{v}_\infty)\right)\to0~~\text{as}~~\omega\to\infty,
\end{align*}
which implies that
\begin{align*}
    (\Tilde{u}_\omega^{\lambda_\omega},\Tilde{u}_\omega^{\lambda_\omega})\to(\Tilde{u}_\infty,\Tilde{v}_\infty)~~\text{strongly in}~~H^1(\mathbb{R}^N)\times H^1(\mathbb{R}^N),~~\text{as}~~\omega\to\infty.
\end{align*}
Now we complete the proof of Lemma \ref{lemma:con}.
\end{proof}

\begin{proof}[\bf Proof of Theorem \ref{Thm:limit}.]
By the definitions of \eqref{def:tilde:u} and \eqref{def:ut} and Lemma \ref{lemma:con}, one can easily obtain
\begin{align*}
\frac{1}{\omega^{\frac{1}{2}}}\left(u_{\omega}\left(\frac{x}{\lambda_\omega\omega^{\frac{1}{2}}}\right),v_{\omega}\left(\frac{x}{\lambda_\omega\omega^{\frac{1}{2}}}\right)\right)\to (\tilde{u}_\infty,\tilde{v}_\infty)~~\text{in}~~H^1(\mathbb{R}^N)\times H^1(\mathbb{R}^N),
\end{align*}
where $\lambda_\omega\to1$ (see \eqref{claim:3}).

On the other hand, by the weak convergence and the norm convergence, we have
    \begin{align*}
\frac{1}{\omega^{\frac{1}{2}}}\left(u_{\omega}\left(\frac{x}{\lambda_\omega\omega^{\frac{1}{2}}}\right),v_{\omega}\left(\frac{x}{\lambda_\omega\omega^{\frac{1}{2}}}\right)\right)\to\frac{1}{\omega^{\frac{1}{2}}}\left(u_{\omega}\left(\frac{x}{\omega^{\frac{1}{2}}}\right),v_{\omega}\left(\frac{x}{\omega^{\frac{1}{2}}}\right)\right)~~\text{in}~~H^1(\mathbb{R}^N)\times H^1(\mathbb{R}^N),
\end{align*}
as $\lambda_\omega\to1$. Hence, combining the above two relations, we conclude the proof of Theorem \ref{Thm:limit}.
\end{proof}

\appendix

\section{Proof of Lemma \ref{Lemma:Pohozaev}}
\begin{proof}[\bf Proof of Lemma \ref{Lemma:Pohozaev}.]
Multiplying the first equation in \eqref{equ:s} by $\bar{\phi}$, the second one by $\bar{\psi}$, integrating
over $\mathbb{R}^N$ and using integration by parts, we have
\begin{align*}
    \int|\nabla \phi|^2+\omega\int|\phi|^2+(ic\cdot\nabla \phi,\phi)-\frac{1}{9}\int|\phi|^2-2\int|\psi|^2|\phi|^2-\frac{1}{3}\Re\int\bar{\phi}^3\psi=0,\\
    \int|\nabla \psi|^2+(3\gamma\omega-3\gamma+\alpha)\int|\psi|^2+(i\gamma c\cdot\nabla \psi,\psi)-9\int|\psi|^2-2\int|\psi|^2|\phi|^2-\frac{1}{9}\Re\int\phi^3\bar{\psi}=0.
\end{align*}
Summing the above identities, we can obtain
\begin{align*}
    &\int|\nabla \phi|^2+\omega\int|\phi|^2+(ic\cdot\nabla \phi,\phi)+\int|\nabla \psi|^2+(3\gamma\omega-3\gamma+\alpha)\int|\psi|^2+(i\gamma c\cdot\nabla \psi,\psi)\notag\\
    =&\int \left(\frac{1}{9}|\phi|^4+9|\psi|^2+4|\psi|^2|\phi|^2+\frac{4}{9}\phi^3\bar{\psi}\right).
\end{align*}
Note that
\begin{align*}
    (x\cdot\nabla u,ic\cdot\nabla u)=-\frac{N-1}{2}(u,ic\cdot\nabla u),~~(x\cdot\nabla u,-\Delta u)=-\frac{N-2}{2}\int|\nabla u|^2,\\
    (x\cdot\nabla u,u)=-\frac{N}{2}\int|u|^2,~~(x\cdot\nabla u,|u|^{p-1}u)=-\frac{N}{p+1}\int|u|^{p+1},
\end{align*}
then multiplying the first equation in \eqref{equ:s} by $x\cdot\nabla\bar{\phi}$, the second one by $x\cdot\nabla\bar{\psi}$, integrating
over $\mathbb{R}^N$, using integration by parts and taking the real part, we deduce that
\begin{align*}
    &-\frac{N-2}{2}\int|\nabla \phi|^2-\omega\frac{N}{2}\int|\phi|^2-\frac{N-1}{2}(\phi,ic\cdot\nabla \phi)\\
    &=-\frac{N}{36}\int|\phi|^4+2\Re\int|\psi|^2\phi x\cdot\nabla\bar{\phi}+\frac{1}{3}\Re\int\bar{\phi}^2\psi x\cdot\nabla\bar{\phi},
\end{align*}
and
\begin{align*}
    &-\frac{N-2}{2}\int|\nabla \psi|^2-(3\gamma\omega-3\gamma+\alpha)\frac{N}{2}\int|\psi|^2-\gamma\frac{N-1}{2}(\psi,ic\cdot\nabla \psi)\\
    =&-\frac{9N}{4}\int|\psi|^4+2\Re\int|\phi|^2\psi x\cdot\nabla\bar{\psi}+\frac{1}{9}\Re\int{\phi}^3x\cdot\nabla\bar{\psi}\\
    =&-\frac{9N}{4}\int|\psi|^4-2\Re\int|\psi|^2\phi x\cdot\nabla\bar{\phi}-N\int|\psi|^2|\phi|^2-\frac{1}{3}\Re\int\phi^2\bar{\psi}x\cdot\nabla\phi-\frac{N}{9}\Re\int\phi^3\bar{\psi}.
\end{align*}
From above two identities, we obtain
\begin{align*}
    &-\frac{N-2}{2}\int\left(|\nabla \phi|^2+|\nabla \psi|^2\right)-\frac{N}{2}\int\left(\omega|\phi|^2+3\gamma\omega|\psi|^2\right)-\frac{N}{2}(\alpha-3\gamma)\int|\psi|^2\\\notag
    &-\frac{N-1}{2}\left[(\phi,ic\cdot\nabla \phi)+\gamma(\psi,ic\cdot\nabla \psi)\right]\\
    =&-\frac{N}{36}\int|\phi|^4-\frac{9N}{4}\int|\psi|^4-N\int|\psi|^2|\phi|^2-\frac{N}{9}\Re\int\phi^3\bar{\psi}.
\end{align*}
This completes the proof of Lemma \ref{Lemma:Pohozaev}.
\end{proof}

In the following, we aim to prove the sharp Gagliardo-Nirenberg-type inequality \eqref{GN:2}. Now, we give the following lemma.
\begin{lemma}
Let $1\leq N\leq3$, system \eqref{equ:elliptic2} has a ground state solution $(Q^*,P^*)$. Moreover, the following identities hold.
\begin{align*}
    &\int|\nabla Q^*|^2+\int|Q^*|^2=\frac{1}{9}\int|Q^*|^4+2\int|Q^*|^2|P^*|^2,\\
    &\int|\nabla P^*|^2+3\gamma\int |P^*|^2=9\int|P^*|^4+2\int|Q^*|^2|P^*|^2,\\
    &\frac{N-2}{2}\int(|\nabla Q^*|^2+|\nabla P^*|^2)+\frac{N}{2}\int|Q^*|^2+\frac{3\gamma N}{2}\int |P^*|^2
    -\frac{N}{4}\int\left(\frac{1}{9}|Q^*|^4+9|P^*|^4+4|Q^*|^2|P^*|^2\right)=0.
\end{align*}
\end{lemma}
\begin{proof}
By the similar argument as \cite{OP2021AMP}, we can obtain this result. Here we omit it.
\end{proof}
\begin{proof}[\bf Proof of inequality \eqref{GN:2}.]
By the Gagliardo-Nirenberg inequality
$\|f\|_{L^4}^4\leq C\|\nabla f\|_{L^2}^2\|f\|_{L^2}^2,$
we have
\begin{align}\notag
    \int\left(\frac{1}{36}|u|^4+\frac{9}{4}|v|^4+|u|^2|v|^2\right)\leq CK(u,v)M_{3\gamma}(u,v).
\end{align}
Now we define
\begin{align}\notag
    J(u,v)=\frac{K(u,v)M_{3\gamma}(u,v)}{\int\left(\frac{1}{36}|u|^4+\frac{9}{4}|v|^4+|u|^2|v|^2\right)}.
\end{align}
The infimum of $J(u,v)$ is attained at a pair of real function $(Q^*,P^*)$, that is
\[\inf J(u,v)=J(Q^*,P^*),\]
if and only if, up to scaling, $(Q^*,P^*)$ is a positive ground state solution of \eqref{equ:elliptic2}. In fact, this is true. A similar argument can be found in \cite{OP2021AMP}. Hence, we can obtain that
$\frac{1}{C_{opt}^{(2)}}=J(Q^*,P^*)=2M_{3\gamma}(Q^*,P^*)$.
\end{proof}

\section{Proof of Proposition \ref{pro}}
In this section, we give the proof of Proposition \ref{pro}.
\begin{lemma}\label{lemma:a}
Let $1\leq N\leq3$, $\gamma>0$ and $\gamma\neq3$. Assume that $\Tilde{a}_{\omega,c,\alpha-3\gamma}\neq \emptyset$, then $\widetilde{\mathcal{G}}_{\omega,c,\alpha-3\gamma}=\Tilde{a}_{\omega,c,\alpha-3\gamma}$. 
 \end{lemma}
 \begin{proof}
Let $(\phi,\psi)\in \Tilde{a}_{\omega,c,\alpha-3\gamma}$. Since $\Tilde{N}_{\omega,c,\alpha-3\gamma}(\phi,\psi)=0$ and $(\phi,\psi)\neq 0$, we have
 \begin{align}\label{exist:1}
     \left(\Tilde{N}^{\prime}_{\omega,c,\alpha-3\gamma}(\phi,\psi),(\phi,\psi)\right)=4\Tilde{\mathcal{Q}}_{\omega,c,\alpha-3\gamma}(\phi,\psi)-16\Tilde{D}(\phi,\psi)=-4\Tilde{\mathcal{Q}}_{\omega,c,\alpha-3\gamma}(\phi,\psi)<0.
 \end{align}
By the Lagrange multiplier theorem, there exists $\lambda\in\mathbb{R}$ such that $\Tilde{S}^{\prime}_{\omega,c,\alpha-3\gamma}(\phi,\psi)=\lambda \Tilde{N}^{\prime}_{\omega,c,\alpha-3\gamma}(\phi,\psi)$. Moreover, we have
 \begin{align}\label{exist:2}
     \lambda\left(\Tilde{N}^{\prime}_{\omega,c,\alpha-3\gamma}(\phi,\psi),(\phi,\psi)\right)=\left(\Tilde{S}^{\prime}_{\omega,c,\alpha-3\gamma}(\phi,\psi),(\phi,\psi)\right)=\Tilde{N}_{\omega,c,\alpha-3\gamma}(\phi,\psi)=0.
 \end{align}
Combining \eqref{exist:1} and \eqref{exist:2}, we can obtain $\lambda=0$. Hence, $\Tilde{S}^{\prime}_{\omega,c,\alpha-3\gamma}(\phi,\psi)=0$, which implies that $(\phi,\psi)\in \Tilde{\mathcal{A}}_{\omega,c,\alpha-3\gamma}$.

Notice that $\Tilde{\mathcal{A}}_{\omega,c,\alpha-3\gamma}\subset \Tilde{\mathcal{N}}_{\omega,c,\alpha-3\gamma}$ and $(\phi,\psi)\in \Tilde{a}_{\omega,c,\alpha-3\gamma}$, we have
\begin{align*}
    \Tilde{S}_{\omega,c,\alpha-3\gamma}(\phi,\psi)\leq \Tilde{S}_{\omega,c,\alpha-3\gamma}(\phi_1,\psi_1)~~\text{for all}~~(\phi_1,\psi_1)\in\Tilde{\mathcal{A}}_{\omega,c,\alpha-3\gamma},
\end{align*}
which implies $(\phi,\psi)\in \Tilde{\mathcal{G}}_{\omega,c,\alpha-3\gamma}$.

On the other hand, let $(\phi,\psi)\in \widetilde{\mathcal{G}}_{\omega,c,\alpha-3\gamma} $. One can take $(\phi_1,\psi_1)\in \Tilde{a}_{\omega,c,\alpha-3\gamma}$, by Lemma \ref{lemma:a}, we have $(\phi_1,\psi_1)\in \widetilde{\mathcal{G}}_{\omega,c,\alpha-3\gamma}$, i.e., $\Tilde{S}_{\omega,c,\alpha-3\gamma}(\phi_1,\psi_1)=\Tilde{S}_{\omega,c,\alpha-3\gamma}(\phi,\psi)$. Therefore, for each $(u,v)\in\Tilde{\mathcal{N}}_{\omega,c,\alpha-3\gamma}$, we obtain
\begin{align*}
    \Tilde{S}_{\omega,c,\alpha-3\gamma}(\phi,\psi)=\Tilde{S}_{\omega,c,\alpha-3\gamma}(\phi_1,\psi_1)\leq \Tilde{S}_{\omega,c,\alpha-3\gamma}(u,v).
\end{align*}
Since $(\phi,\psi)\in \widetilde{\mathcal{G}}_{\omega,c,\alpha-3\gamma}\subset\Tilde{\mathcal{N}}_{\omega,c,\alpha-3\gamma}$, we deduce that $(\phi,\psi)\in \Tilde{a}_{\omega,c,\alpha-3\gamma}$.
\end{proof}

\begin{lemma}\label{lemma:positive}
Let $1\leq N\leq3$, $\gamma>0$ and $\gamma\neq3$. Then there exists $M>0$ such that $\Tilde{\mu}_{\omega,c,\alpha-3\gamma}={\mu}_{\omega,c,\alpha-3\gamma}\geq\frac{1}{2M}$.
\end{lemma}

\begin{proof}
By the definition of $\Tilde{S}_{\omega,c,\alpha-3\gamma}$,  $\Tilde{\mathcal{Q}}_{\omega,c,\alpha-3\gamma}$ and $\Tilde{N}_{\omega,c,\alpha-3\gamma}$, we have the following relation
\[\Tilde{S}_{\omega,c,\alpha-3\gamma}=\frac{1}{2}\Tilde{\mathcal{Q}}_{\omega,c,\alpha-3\gamma}+\frac{1}{4}\Tilde{N}_{\omega,c,\alpha-3\gamma}.\]
One can rewrite  $\Tilde{\mu}_{\omega,c,\alpha-3\gamma}$ as
\begin{align}\label{exist:3}
    \Tilde{\mu}_{\omega,c,\alpha-3\gamma}=\inf\left\{\frac{1}{2}\Tilde{\mathcal{Q}}_{\omega,c,\alpha-3\gamma}:~~(\phi,\psi)\in\Tilde{\mathcal{N}}_{\omega,c,\alpha-3\gamma}\right\}.
\end{align}
By the H\"older inequality, Gagliardo-Nirenberg inequality and Sobolev embedding theorem, we have
\begin{align*}
    &\Tilde{\mathcal{Q}}_{\omega,c,\alpha-3\gamma}(u,v)=2\widetilde{D}(u,v)\\
    =&2\int \left(\frac{1}{36}|u(x)|^4+\frac{9}{4}|v(x)|^4+|u(x)|^2|v(x)|^2+\frac{1}{9}\Re\left[ e^{\frac{i(3-\gamma)}{2}c\cdot x}\left(\bar{u}^3(x)v(x)\right)\right]\right)\\
    \leq&M_1\|\nabla u\|_{L^2}^N\|u\|_{L^2}^{4-N}+M_1\|\nabla v\|_{L^2}^N\|v\|_{L^2}^{4-N}+M_3\|\nabla u\|_{L^2}^{\frac{3N}{4}}\|u\|_{L^2}^{\frac{3(4-N)}{4}}\|\nabla v\|_{L^2}^{\frac{N}{4}}\|v\|_{L^2}^{\frac{(4-N)}{4}}\\
    \leq&M\Tilde{\mathcal{Q}}^2_{\omega,c,\alpha-3\gamma}(u,v),
\end{align*}
where $M=M(N,\gamma,c)>0$.
Dividing both sides of the above by $\Tilde{\mathcal{Q}}_{\omega,c,\alpha-3\gamma}>0$, we obtain the positive lower bound \[\Tilde{\mathcal{Q}}_{\omega,c,\alpha-3\gamma}\geq\frac{1}{M}.\]
This means that
$\mu_{\omega,c,\alpha-3\gamma}=\Tilde{\mu}_{\omega,c,\alpha-3\gamma}\geq\frac{1}{2M}$. This completes the proof of Lemma \ref{lemma:positive}.
\end{proof}

\begin{lemma}\label{lemma:upper}
Let  $1\leq N\leq3$, $\gamma>0$ and $\gamma\neq3$. If $(u,v)\in \Tilde{X}_{\omega,c,\alpha-3\gamma}$ satisfies $\Tilde{N}_{\omega,c,\alpha-3\gamma}(u,v)<0$, then $\frac{1}{2}\Tilde{\mathcal{Q}}_{\omega,c,\alpha-3\gamma}>\Tilde{\mu}_{\omega,c,\alpha-3\gamma}$.
\end{lemma}
\begin{proof}
If $\Tilde{N}_{\omega,c,\alpha-3\gamma}(u,v)<0$, then $2\Tilde{D}(u,v)>\Tilde{\mathcal{Q}}_{\omega,c,\alpha-3\gamma}(u,v)>0$. From this, we have
\begin{align*}
    \lambda_0^2:=\frac{\Tilde{\mathcal{Q}}_{\omega,c,\alpha-3\gamma}(u,v)}{2\Tilde{D}(u,v)}\in(0,1)
\end{align*}
and $\Tilde{N}_{\omega,c,\alpha-3\gamma}(\lambda_0u,\lambda_0v)=0$. By \eqref{exist:3}, we deduce
\begin{align*}
    \Tilde{\mu}_{\omega,c,\alpha-3\gamma}\leq\frac{1}{2}\Tilde{\mathcal{Q}}_{\omega,c,\alpha-3\gamma}(\lambda_0u,\lambda_0v)=\frac{\lambda_0^2}{2}\Tilde{\mathcal{Q}}_{\omega,c,\alpha-3\gamma}(u,v)<\frac{1}{2}\Tilde{\mathcal{Q}}_{\omega,c,\alpha-3\gamma}(u,v).
\end{align*}
Now we complete the proof of this Lemma.
\end{proof}

Note that $\Tilde{\mathcal{Q}}_{\omega,c,\alpha-3\gamma}$ and $\Tilde{D}$ are invariant under
\begin{align}\label{def:tau}
    \Tilde{\tau}_{y}(u,v):=\left(e^{-\frac{i}{2}c\cdot x}u(\cdot-y),e^{-\frac{i\gamma}{2}c\cdot x}v(\cdot-y)\right),
\end{align}
that is,
\[\Tilde{\mathcal{Q}}_{\omega,c,\alpha-3\gamma}( \Tilde{\tau}_{y}(u,v))=\Tilde{\mathcal{Q}}_{\omega,c,\alpha-3\gamma}(u,v),~~\Tilde{D}( \Tilde{\tau}_{y}(u,v))=\Tilde{D}(u,v),\]
for all $y\in\mathbb{R}^N$.

\begin{lemma}\label{lemma:w}
Let $1\leq N\leq3$, $\gamma>0$ and $\gamma\neq3$.  If a sequence $\{(u_n,v_n)\}$ in $\Tilde{X}_{\omega,c,\alpha-3\gamma}$ satisfies
\begin{align*}
    \Tilde{\mathcal{Q}}_{\omega,c,\alpha-3\gamma}(u_n,v_n)\to A_1,~~\Tilde{D}(u_n,v_n)\to A_2~~\text{as}~~n\to\infty,
\end{align*}
for some positive constants $A_1,A_2>0$, then there exist $\{y_n\}$ and $(u,v)\in\Tilde{X}_{\omega,c,\alpha-3\gamma}\setminus \{(0,0)\}$ such that $\{\Tilde{\tau}_{y}(u_n,v_n)\}$ has a subsequence that weakly converges to $(u,v)$ in $\Tilde{X}_{\omega,c,\alpha-3\gamma}$. Moreover,
\begin{align*}
    \Tilde{D}(\Tilde{\tau}_{y}(u_n,v_n))-\Tilde{D}(\Tilde{\tau}_{y}u_n-u,\Tilde{\tau}_{y}v_n-v)\to \Tilde{D}(u,v)~~\text{as}~~n\to\infty.
\end{align*}
\end{lemma}
\begin{proof}
Since $\Tilde{\mathcal{Q}}_{\omega,c,\alpha-3\gamma}(u_n,v_n)\to A_1$, we deduce that the sequence $\{(u_n,v_n)\}$ is bounded in $\Tilde{X}_{\omega,c,\alpha-3\gamma}$. Since $\Tilde{D}(u_n,v_n)\to A_2$, we obtain
\begin{align*}
    \limsup_{n\to\infty}\|u_n\|_{L^4}+\limsup_{n\to\infty}\|v_n\|_{L^4}>0.
\end{align*}
Hence, by \cite[Lemma 6]{L1983Invent}, we can obtain the desired result.
\end{proof}

\begin{lemma}\label{lemma:strong}
Let  $1\leq N\leq3$, $\gamma>0$ and $\gamma\neq3$.  If a sequence  $\{(u_n,v_n)\}$ in $\Tilde{X}_{\omega,c,\alpha-3\gamma}$ satisfies
\begin{align*}
    \Tilde{N}_{\omega,c,\alpha-3\gamma}(u_n,v_n)\to0,~~\Tilde{S}_{\omega,c,\alpha-3\gamma}(u_n,v_n)\to\Tilde{\mu}_{\omega,c,\alpha-3\gamma}~~\text{as}~~n\to\infty,
\end{align*}
then there exist $\{y_n\}$ and $(u,v)\in\Tilde{X}_{\omega,c,\alpha-3\gamma}\setminus\{(0,0)\}$ such that $\{\Tilde{\tau}_{y_n}(u_n,v_n)\}$ has a subsequence that converges to $(u,v)$ in $\Tilde{X}_{\omega,c,\alpha-3\gamma}$. In particular, $(u,v)\in\Tilde{a}_{\omega,c,\alpha-3\gamma}$.
\end{lemma}
\begin{proof}
By assumptions, we have
\begin{align*}
    \frac{1}{2}\Tilde{\mathcal{Q}}_{\omega,c,\alpha-3\gamma}(u_n,v_n)=\Tilde{S}_{\omega,c,\alpha-3\gamma}(u_n,v_n)-\frac{1}{4}\Tilde{N}(u_n,v_n)\to\Tilde{\mu}_{\omega,c,\alpha-3\gamma},\\
   \Tilde{D}(u_n,v_n)=\Tilde{S}_{\omega,c,\alpha-3\gamma}(u_n,v_n)-\frac{1}{2}\Tilde{N}_{\omega,c,\alpha-3\gamma}(u_n,v_n)\to \Tilde{\mu}_{\omega,c,\alpha-3\gamma}.
\end{align*}
From Lemma \ref{lemma:positive} and Lemma \ref{lemma:w}, there exist $\{y_n\}$, $(u,v)\in \Tilde{X}_{\omega,c,\alpha-3\gamma}\setminus\{(0,0)\}$, and a subsequence of $\Tilde{\tau}_{y_n}(u_n,v_n)$(still denoted by $\Tilde{\tau}_{y_n}(u_n,v_n)$) such that $\Tilde{\tau}_{y_n}(u_n,v_n)\rightharpoonup (u,v)$ weakly in $\Tilde{X}_{\omega,c,\alpha-3\gamma}$.

By the weak convergence of $\Tilde{\tau}_{y_n}(u_n,v_n)$ and Lemma \ref{lemma:w}, we have
\begin{align}\label{conv:1}
    \Tilde{\mathcal{Q}}_{\omega,c,\alpha-3\gamma}(\Tilde{\tau}_{y_n}(u_n,v_n))-\Tilde{\mathcal{Q}}_{\omega,c,\alpha-3\gamma}\left(\Tilde{\tau}_{y_n}(u_n,v_n)-(u,v)\right)\to\Tilde{\mathcal{Q}}_{\omega,c,\alpha-3\gamma}(u,v),\\\label{conv:2}
    \Tilde{N}_{\omega,c,\alpha-3\gamma}(\Tilde{\tau}_{y_n}(u_n,v_n))-\Tilde{N}_{\omega,c,\alpha-3\gamma}\left(\Tilde{\tau}_{y_n}(u_n,v_n)-(u,v)\right)\to\Tilde{N}_{\omega,c,\alpha-3\gamma}(u,v).
\end{align}
From \eqref{conv:1} and $\Tilde{\mathcal{Q}}_{\omega,c,\alpha-3\gamma}(u,v)>0$, we obtain that, up to subsequence,
\begin{align*}
    \frac{1}{2}\lim_{n\to\infty}\Tilde{\mathcal{Q}}_{\omega,c,\alpha-3\gamma}\left(\Tilde{\tau}_{y_n}(u_n,v_n)-(u,v)\right)<\frac{1}{2}\lim_{n\to\infty}\Tilde{\mathcal{Q}}_{\omega,c,\alpha-3\gamma}(\Tilde{\tau}_{y_n}(u_n,v_n))=\Tilde{\mu}_{\omega,c,\alpha-3\gamma}.
\end{align*}
Combining this and Lemma \ref{lemma:upper}, we obtain $\Tilde{N}_{\omega,c,\alpha-3\gamma}\left(\Tilde{\tau}_{y_n}(u_n,v_n)-(u,v)\right)>0$ for $n$ large enough. Therefore, from \eqref{conv:2}, $\Tilde{N}_{\omega,c,\alpha-3\gamma}(u,v)\leq0$ since $\Tilde{N}_{\omega,c,\alpha-3\gamma}(\Tilde{\tau}_{y_n}(u_n,v_n))\to0$. Again, by Lemma \ref{lemma:upper} and the weakly lower semi-continuity of norms, we have
\begin{align*}
    \Tilde{\mu}_{\omega,c,\alpha-3\gamma}\leq\frac{1}{2}\Tilde{\mathcal{Q}}_{\omega,c,\alpha-3\gamma}(u,v)\leq\lim_{n\to\infty}\frac{1}{2}\Tilde{\mathcal{Q}}_{\omega,c,\alpha-3\gamma}\left(\Tilde{\tau}_{y_n}(u_n,v_n)\right)=\Tilde{\mu}_{\omega,c,\alpha-3\gamma}.
\end{align*}
Therefore, by \eqref{conv:1}, we deduce $\Tilde{Q}_{\omega,c,\alpha-3\gamma}\left(\Tilde{\tau}_{y_n}(u_n,v_n)-(u,v)\right)\to0$ as $n\to\infty$, which implies that $\Tilde{\tau}_{y_n}(u_n,v_n)\to(u,v)$ strongly in $\Tilde{X}_{\omega,c,\alpha-3\gamma}$. Now we have completed the proof of this lemma.
\end{proof}
\begin{proof}[\bf Proof of Proposition \ref{pro}.]
Combining the Lemmas \ref{lemma:a}, and \ref{lemma:strong}, one can get Proposition \ref{pro}.
\end{proof}




\renewcommand{\proofname}{\bf Proof.}

\noindent
{\bf Acknowledgments.}

 This work was
 supported by  the National Natural Science
Foundation of China (12301090), Natural Science Foundation of Zhejiang Province (LMS26A010005, LZ26A010002) and  Fundamental Research Program of Shanxi Province (202203021222126).

\vspace{0.2cm}
\noindent
{\bf Data Availability}

We do not analyse or generate any datasets, because our work proceeds within a theoretical
and mathematical approach.

\vspace{0.2cm}
\noindent
{\bf Conflict of interest}

The authors declare no potential conflict of interest.


\vspace*{.5cm}



\bigskip






\end{document}